\documentclass{article} 
\usepackage{setspace, ifdraft, blindtext}
\usepackage{authblk}
\usepackage{enumerate}
\usepackage{kantlipsum}
\usepackage{tikz}
\usepackage{pgf,tikz}
\usetikzlibrary{arrows}
\usetikzlibrary[patterns]
\usepackage{graphicx}
\usepackage{caption}
\usepackage{subcaption}
\usepackage{graphicx}
\usepackage{wrapfig}

\usepackage{amsmath, amssymb, amsthm}
\usepackage{color}
\usepackage{kantlipsum}
\usepackage{pgf,tikz}
\usepackage{mathrsfs}
\usepackage{tikz}
\usepackage{pgf,tikz}
\usetikzlibrary[patterns]
\usepackage{graphicx}
\usepackage{caption}
\usepackage{subcaption}
\usepackage{algorithmicx}
\usepackage{algorithm}
\usepackage[noend]{algpseudocode}
\usepackage[margin=1in]{geometry}

\usepackage{amsmath, amssymb, amsthm}
\usepackage{color}

\usepackage{xr}

\makeatletter
\def\BState{\State\hskip-\ALG@thistlm}
\makeatother

\theoremstyle{plain}
\newtheorem{thm}{Theorem}[section]

\newtheorem*{prop*}{Proposition}
\newtheorem{lem}[thm]{Lemma}
\newtheorem*{lem*}{Lemma}
\newtheorem{dfn}[thm]{Definition}
\newtheorem{cor}[thm]{Corollary}
\newtheorem{claim}[thm]{Claim}

\newtheorem*{rem}{Remark}

\newcommand{\ta}{\tilde \alpha}
\newcommand{\dist}{\text{dist}}
\newcommand{\PP}{\mathcal{P}}
\newcommand{\cS}{\mathcal{S}}
\newcommand{\cF}{\mathcal{F}}

\newcommand{\Active}{\textsc{Active}}
\newcommand{\Frozen}{\textsc{Frozen}}
\newcommand{\Mature}{\textsc{Mature}}
\newcommand{\Rich}{\textsc{Rich}}
\newcommand{\Can}{\textsc{Candidate}}

\newcommand{\NR}{\textsc{NResolve}}
\newcommand{\D}{D} 

\newcommand{\pn}{\rho_\text{EMD}}
\newcommand{\pc}{\rho_\text{r}}
\newcommand{\ps}{\rho_\text{s}}
\newcommand{\pd}{\rho_\text{d}}
\newcommand{\pds}{\rho_{\text{ds}}}
\newcommand{\pkt}{\rho_{\text{KT}}}
\newcommand{\ppt}{\rho_{\text{PT}}} 

\newcommand{\beq}{\begin{equation}}
\newcommand{\eeq}{\end{equation}}

\date{}
\title{\vspace{-0.7cm} Permutation Testing}

\begin{document}
\definecolor{ffqqqq}{rgb}{1.,0.,0.}
\definecolor{ttqqqq}{rgb}{0.2,0.,0.}

\makeatletter
\algrenewcommand\algorithmiccomment[2][\normalsize]{{#1\hfill\(\triangleright\) #2}}

\makeatother

\date{\today}
\title{\vspace{-0.7cm} Fast property testing and metrics for permutations}

\author[1]{Jacob Fox\thanks{Department of Mathematics, Stanford University, Stanford, CA 94305. Email: {\tt jacobfox@stanford.edu}. Research partially supported by a Packard Fellowship, by NSF Career Award DMS-1352121 and by an Alfred P. Sloan Fellowship.} \and Fan Wei\thanks{Department of Mathematics, Stanford University, Stanford, CA 94305. Email: {\tt fanwei@stanford.edu}.}}

\maketitle

\begin{abstract}
The goal of property testing is to quickly distinguish between objects which satisfy a property and objects that are $\epsilon$-far from satisfying the property. There are now several general results in this area which show that natural properties of combinatorial objects can be tested with ``constant'' query complexity, depending only on $\epsilon$ and the property, and not on the size of the object being tested. The upper bound on the query complexity coming from the proof techniques is often enormous and impractical. It remains a major open problem if better bounds hold.

Maybe surprisingly, for testing with respect to the rectangular distance, we prove there is a universal (not depending on the property), polynomial in $1/\epsilon$ query complexity bound for two-sided testing hereditary properties of sufficiently large permutations. We further give a nearly linear bound with respect to a closely related metric which also depends on the smallest forbidden subpermutation for the property. Finally, we show that several different permutation metrics of interest are related to the rectangular distance, yielding similar results for testing with respect to these metrics.

\medskip

\end{abstract}

\section{Introduction}

Traditionally, algorithms that run in time polynomial in the input size were considered fast. However, as the desired input size has increased, this notion of fast is sometimes insufficient. Some examples include in algorithmic problems on networks like the internet or the brain, or in ranking websites for search algorithms, in which the structures being studied have billions of elements and are often not well understood. In order to handle such large structures, sublinear time algorithms are desired. One would not expect for such algorithms to be able to determine properties of the structures with certainty. This is where property testing comes in. 

The goal of property testing is to quickly distinguish between objects which satisfy a property and objects that are $\epsilon$-far from satisfying the property. The study of this notion was initiated by Rubinfield and Sudan \cite{RS}.  Subsequently, Goldreich, Goldwasser, and Ron \cite{GGR} began the investigation of property testers for combinatorial objects. There are now several quite general results in this area which show that properties can be tested with ``constant'' query complexity, depending only on $\epsilon$ and the property, and not on the size of the object being tested. A property $P$ is {\it one-sided testable} if there is a function $q(\epsilon)$ and a randomized algorithm with query complexity $q(\epsilon)$ which, on an input which has property $P$, correctly outputs that the object has property $P$, and on input that is $\epsilon$-far from satisfying $P$, correctly outputs with probability at least $2/3$ that the object is $\epsilon$-far from satisfying $P$. Property $P$ is {\it two-sided testable} if it correctly outputs in either case with probability at least $2/3$. Note that if an input neither satisfies $P$ nor is $\epsilon$-far from $P$, it has no guarantee on the output. An exemplary result in this area, due to Alon and Shapira \cite{AS08}, states that every hereditary graph property is one-sided testable. However the query complexity bound it gives is at least of wowzer-type in $1/\epsilon$, which is one level higher in the Ackermann hierarchy than the tower function, as it uses the strong regularity lemma. Enormous bounds like this on the query complexity are typical of many general results in this area and is a major drawback as the bounds are impractical. Besides regularity methods, which often give tower-type or worse estimates, some of the other results in this area are established using compactness arguments and thus yield no bound. There are some examples of progress on quantitative bounds for property testing. See, for example, \cite{Alon,AlSh,CF1,CF2,FoLo}. However, it remains a major open problem if better bounds hold for the various property testing results.

In this paper, we address this problem for permutations. To properly understand our results, it is important to try to first determine what is a good notion (or notions) of distance between combinatorial objects. This is because we need to understand which metric we are using when we say that two objects are $\epsilon$-far from each other in this metric. For graphs, the edit distance, which is the fraction of pairs which one needs to add or delete edges from in order to turn one graph into the other, is quite natural, and it is not surprising that it is the most studied with regards to graph property testing. For permutations, there are now several important metrics that naturally arise in ranking problems in statistics. See, for example, the book by Diaconis \cite{D88}. This makes it less clear for which metrics permutation property testing should be done with respect to. 

An early paper of Cooper \cite{Cooper} develops an analogue of Szemer\'edi's regularity lemma for permutations\footnote{This regularity lemma was subsequently improved upon by Hoppen, Kohayakawa, and Sampaio \cite{HKS12} and with much better quantitative estimates by Fox, Lov\'asz, and Zhao \cite{FLZ}.} and deduces a permutation removal lemma. Typically, removal lemmas are equivalent to saying that certain properties of combinatorial objects are one-sided testable with respect to some metric. However, Cooper's permutation removal lemma does not give such a metric and so does not actually translate to a result in property testing. 

Since Cooper's work, there are now two different notions of distance which have been studied for permutation property testing, the rectangular distance (or cut distance), and Kendall's tau distance. The rectangular distance for permutations is an analogue of the Frieze--Kannan cut distance for graphs, which has played an important role in the development of the weak regularity lemma, graph limits, and approximation algorithms for graphs. See, for example, the book by Lov\'asz \cite{Lov} and the paper \cite{FLZ}. 

A permutation $\pi$ of length $n$ is a bijection from $[n]:=\{1,2,\ldots,n\}$ to itself. We can represent $\pi$ as $n$ points in the plane with the coordinates of the $i^{\textrm{th}}$ point being $(i, \pi(i))$. The \emph{rectangular (cut) distance} between two permutations $\pi_1,\pi_2$ of length $n$ is defined to be 
\[\pc(\pi_1, \pi_2) =\frac{1}{n} \max_{S, T} \left| \left|\pi_1(S) \cap T \right|-\left| \pi_2(S)\cap T \right| \right |,\]
where the maximum is over all subintervals $S,T$ of $[n]$. Thus, the rectangular distance is the normalized maximum discrepancy in rectangles between the number of points of the form $(i,\pi_1(i))$ and the number of points of the form $(i,\pi_2(i))$. While the rectangular distance is defined globally, through a counting lemma, it can be shown that two permutations have small rectangular distance if and only if they have roughly the same densities of all small subpermutations. This is an analogue of similar results for graphs; see \cite{Lov} and \cite{HKMRS} for details on these results for graphs and permutations, respectively. 

A copy of a permutation $\sigma$ of length $k$ in a permutation $\pi$ of length $n$ is a subsequence of $\pi$ that has the same order type as $\sigma$. That is, a copy of $\sigma$ in $\pi$ is a sequence of integers $1 \leq i_1 < i_2 <\ldots < i_k \leq n$ such that $\pi(i_j) < \pi(i_{\ell})$ if and only if $\sigma(j) < \sigma(\ell)$. If $\pi$ contains a copy of $\sigma$, then we say that $\sigma$ is a \emph{subpermutation} of $\pi$. If $\pi$ does not contain a copy of $\sigma$, then we say that $\pi$ avoids $\sigma$ or is $\sigma$-free. A permutation property is just a family of permutations. A permutation property $\mathcal{P}$ is {\it hereditary} if it is closed under subpermutations, that is, if every subpermutation of a permutation in $\mathcal{P}$ is also in $\mathcal{P}$. Hoppen, Kohayakawa, Moreira, and Sampaio \cite{HKMS} proved that every hereditary permutation property is one-sided testable with respect to the rectangular distance. Their proof uses a compactness argument and does not give any bound on the query complexity.  They also conjectured a stronger result that hereditary permutation properties are strongly testable, i.e., can be tested with respect to \emph{Kendall's tau distance}: for two permutations $\pi_1, \pi_2$ of length $n$, \[
\pkt(\pi_1, \pi_2) = \frac{1}{\binom{n}{2}} \left|  \{ (i,j) \text{ such that } \pi_1(i) < \pi_1(j), \pi_2(i) > \pi_2(j), i,j \in [n]    \}  \right|.\]
Alternatively, Kendall's tau distance between $\pi_1, \pi_2$ can also be defined as the minimum number of adjacent transpositions (i.e., swapping the $y$-coordinates of the points $(i, \pi(i))$ and $(i+1, \pi(i+1))$) required to turn $\pi_1$ into $\pi_2$, and normalized by dividing by $\binom{n}{2}$. 
This conjecture is stronger because the rectangular distance is small if Kendall's tau distance is small, but the converse is not true. For example, for two random permutations of length $n$ almost surely have rectangular distance $o(1)$, but Kendall's tau distance $\Omega(1)$. The nice conjecture of Hoppen et al.~was verified by Klimo\v sov\'a and Kr\'al' \cite{KK}. However, even for the property of being $\sigma$-free for some fixed permutation $\sigma$, the bound on the query complexity is enormous, of Ackermann-type in $1/\epsilon$, and hence not primitive recursive.\footnote{In the conference version of \cite{KK}, it is incorrectly stated that the proof gives a double exponential bound for testing $\sigma$-freeness.} In another work \cite{FW2}, we prove that there is a polynomial in $1/\epsilon$ bound for one-sided testing $\sigma$-freeness, where the exponent depends on $\sigma$. The result generalizes to show that hereditary properties are one-sided testable with respect to Kendall's tau distance, and for typical properties, it gives a polynomial bound.

Another important permutation metric is Spearman's footrule distance. 
For two permutation $\pi_1, \pi_2: [n] \to [n]$, their \emph{Spearman's footrule distance} is 
\[ \D(\pi_1,\pi_2) = \frac{1}{\binom{n}{2}}\sum_{i=1}^n |\pi_1(i) - \pi_2(i)|.\]
 A fundamental result of Diaconis and Graham \cite{DG} states that
\[\pkt(\pi_1,\pi_2 ) \leq \D(\pi_1,\pi_2) \leq 2\pkt(\pi_1,\pi_2).\] 
Thus Kendall's tau distance and Spearman's footrule distance are within a factor of two, and so testing with respect to Kendall's tau distance is essentially equivalent to testing with respect to Spearman's footrule distance. 

Maybe surprisingly, for testing with respect to the rectangular distance, we prove that there is a universal (not depending on the property), polynomial in $1/\epsilon$ query complexity bound for two-sided testing of hereditary properties of sufficiently large permutations. One drawback of the definition of the rectangular distance is that it is global, whereas with Kendall's tau distance, we see that we can make sequential local moves in order to get from one permutation to the other. We study a new distance for permutations, whose general form is called \emph{earth mover's distance} in statistics. It turns out to be quite natural and defined based on local moves, yet we prove it is small if and only if the rectangular distance is small. 

\begin{dfn}[Planar tau distance]
We say that $\sigma$ is obtained from $\pi$ by a \emph{planar simple transposition}, if there is an integer $1 \leq i < |\pi|-1$ such that $\sigma$ is the same as $\pi$ except $\sigma(i) = \pi(i+1), \sigma(i+1) = \pi(i)$, or $\sigma^{-1}(i) = \pi^{-1}(i+1), \sigma^{-1}(i+1) = \pi^{-1}(i)$. The \emph{planar tau distance} $ \ppt(\pi_1, \pi_2)$ between two permutations $\pi_1, \pi_2$ of length $n$ is defined as $\frac{1}{{n \choose 2}}$ times the minimum number of planar simple transpositions needed to transform $\pi_1$ into $\pi_2$.
\end{dfn} 
The $\frac{1}{{n \choose 2}}$ factor in the definition is the proper normalization in order to guarantee that this distance is always at most one. We can also define the planar tau distance in terms of Kendall's tau distance, since the planar tau distance allows for adjacent transpositions in both the horizontal and vertical directions. Thus 
\[ \ppt(\pi_1, \pi_2) = \min_{\text{bijection}~\theta:[n]\to[n]}  \pkt(\text{id}, \theta) + \pkt(\pi_1,\pi_2 \circ \theta).\]

It is also useful and interesting to define a planar analogue of Spearman's footrule distance.
\begin{dfn}[Earth mover's distance (for permutations)]
The \emph{earth mover's distance} between two permutations $\pi_1, \pi_2$ of length $n$ is defined as 
\[ \pn(\pi_1, \pi_2) =\frac{1}{\binom{n}{2}} \min_{\text{bijection}~\theta:[n] \to [n]} \left(\sum_{i=1}^n|i-\theta(i)| + |\pi_1(i) - \pi_2(\theta(i))|\right),\]
 where the minimum is over all bijections $\theta: [n] \to [n]$.
\end{dfn}
 This is the sum of $L_1$ distances between a point in $\pi_1$ and the point in $\pi_2$ that it maps to under $\theta$.
We can treat $\theta$ as a permutation. 
Thus the earth mover's distance is equivalent to 
\[ \pn(\pi_1, \pi_2) = \min_{\text{bijection}~\theta:[n]\to[n]} \D(\text{id}, \theta) + \D(\pi_1,\pi_2 \circ \theta) .
\]
Taking $\theta$ to be the identity permutation $\text{id}$, we obtain $\ppt(\pi_1, \pi_2) \leq \pkt(\pi_1, \pi_2)$ and $\pn(\pi_1, \pi_2) \leq D(\pi_1, \pi_2)$. Thus, the planar metrics are at most their classical analogues. 

The earth mover's distance is a special case of more general metrics that have been extensively studied before in other contexts. It is called the earth mover's distance or the Monge-Kantorovich norm in computer science and was first introduced by Monge  \cite{Monge} in 1781 as a central concept in transportation. It is a natural way of measuring the similarity between two digital images  (see, e.g., \cite{RTS}). In the case of permutations, the digital image has a single one in each row and column. In analysis, it is known as the Wasserstein metric. It is also a special case of the minimum weighted matching problem (see, e.g., \cite{V}). 

By the definitions of the planar tau distance through Kendall's tau distance and the earth mover's distance through 
Spearman's footrule distance, and by the Diaconis-Graham inequality, we therefore get the following planar analogue of the Diaconis-Graham inequality, which was pointed out by Diaconis \cite{Diaconis}.
\begin{cor}\label{planarfactortwo}
\[\ppt(\pi_1,\pi_2) \leq \pn(\pi_1, \pi_2)  \leq 2\ppt(\pi_1, \pi_2). \]
\end{cor}

The following result shows that the planar tau distance is small if and only if the rectangular distance is small. Together with the previous result, it shows that testing with respect to any of these metrics is the same up to a quadratic change in the testing parameter $\epsilon$. 
\begin{thm}\label{thmquadraticmetric}
For any two permutations $\pi_1, \pi_2$ of length $n$, we have 
\[  \frac{1}{8} \pc(\pi_1, \pi_2)^2 \leq \pn(\pi_1, \pi_2)  \leq 48 \pc(\pi_1, \pi_2)^{1/2}.  \]
\end{thm}
Thus the new planar metrics share many of the advantages of both the rectangular distance and Kendall's tau distance. This result and other results relating permutation metrics are proved in the full version of this paper. 

\begin{dfn}[Blow-up of a permutation]\label{def:blowup}
A permutation $\alpha'$ is a \emph{blow-up} of another permutation $\alpha$ if and only if we can find positive integers $ 1 = k_1< k_2<  \dots < k_{|\alpha|+1}=|\alpha'|+1$ with the following property. If $i_1, i_2 \leq |\alpha'|$ satisfy that $k_{j_1} \leq i_1 < k_{j_1+1}$ and $k_{j_2} \leq i_2 < k_{j_2+1}$ with $j_1 \neq j_2$, then $\alpha'(i_1) > \alpha'(i_2)$ if and only if $\alpha(j_1) > \alpha(j_2)$.
\end{dfn}
Intuitively, it means that each point $(i, \alpha(i))$ blows up into a \emph{block} $\{ (i', \alpha'(i')), k_i \leq i' < k_{i+1}\}.$ The $i$-th block is of size $k_{i+1} - k_i$. Notice that we did not specify the permutation within each block. Figure \ref{fig:blowup} is an example of a blow-up $\alpha'$ of a permutation $\alpha$.  

\begin{figure}[H]
\centering
\caption{Blow-up of a permutation. \footnotesize{In this example, $\alpha'$ (the permutation in black) is a blow-up of $\alpha$ (the permutation in red), with the four blocks being four gray squares. The first, second, third, and fourth points in $\alpha$ blow-up into blocks of sizes three, four, two, and one respectively.}}
\includegraphics[scale=0.4]{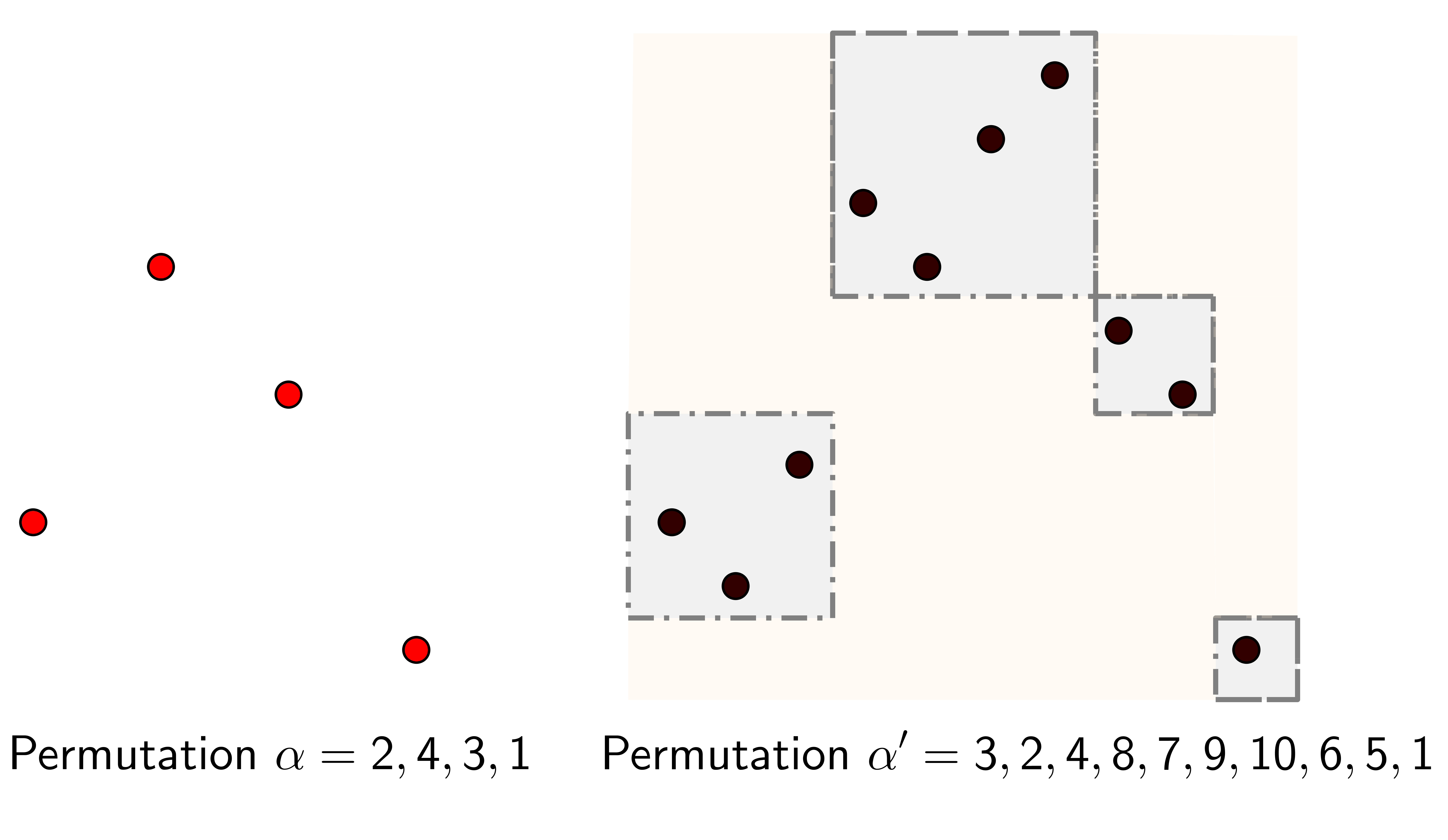}
\label{fig:blowup}
\end{figure}

A \emph{$k$-blow-up} of $\alpha$ is a blow-up of $\alpha$ where each block has size $k$. We next define some important parameters for the property $\PP$ related to blow-ups of permutations. Since we work with a single property $\PP$, we leave it out of the notation to make the notation simpler. 
\begin{dfn}[Blow-up parameter for $\PP$]\label{def:blowupp}
Given a permutation $\alpha$,  let $k^*(\alpha)$ be the minimum positive integer $k$ (if it exists) such that no $k$-blow-up of $\alpha$ is in $\PP$. If no such integer $k$ exists, i.e., for every $k$ there is a $k$-blow-up of $\alpha$ which is in $\PP$, then we define  $k^*(\alpha)=\infty$. 
Given a positive integer $T$, let $k^*(T)$ be the maximum of $k^*(\alpha)$ over all 
permutations $\alpha$ of length $T$ for which $k^*(\alpha)<\infty$. If no such $\alpha$ exists, i.e., $k^*(\alpha)=\infty$ for all permutations $\alpha$ of length $T$, then we define $k^*(T)=\infty$. 
\end{dfn}
Note that if $\PP$ is hereditary and has a forbidden subpermutation of length at most $T$, then it has a forbidden subpermutation $\alpha$ of length $T$, and $\alpha$ is a $1$-blow-up of itself and is not in $\PP$, which implies that $k^*(T)$ is finite in this case.  

We give a nearly linear bound for the query complexity for testing $\PP$ which depends on the smallest forbidden subpermutation for the property. We first need an important definition. A $0-1$ matrix $A$ {\it contains} another $0-1$ matrix $B$ if there is a submatrix $A'$ of $A$ of the same size as $B$ such that for every one entry of $B$, the corresponding entry of $A'$ is a one. For a permutation $\sigma$, the extremal number $\textrm{ex}(n,\sigma)$ is the maximum number of one entries in a $n \times n$ matrix with entries $0$ or $1$ which does not contain the permutation matrix of $\sigma$. F\"uredi and Hajnal \cite{FH92} conjectured that for each permutation $\sigma$, the limit $c(\sigma):=\lim_{n \to \infty} \frac{\textrm{ex}(n,\sigma)}{n}$ exists. Klazar \cite{Kl00} proved that the F\"uredi-Hajnal conjecture implies the well-known Stanley-Wilf conjecture. A celebrated result of Marcus and Tardos \cite{MT} verifies the F\"uredi-Hajnal conjecture, and hence the Stanley-Wilf conjecture. It shows that $c(\sigma)=2^{O(|\sigma| \log |\sigma|)}$. The first author \cite{F} improved the bound to $c(\sigma)=2^{O(|\sigma|)}$, and showed that  $c(\sigma)=2^{|\sigma|^{\Omega(1)}}$ for almost all permutations $\sigma$ of a given order.  The constant $c(\sigma)$ is known as the {\it F\"uredi-Hajnal constant} of $\sigma$. The fact that $\textrm{ex}(n,\sigma)$ is superadditive in $n$ implies that $\textrm{ex}(n,\sigma) \leq c(\sigma)n$ for all $n$. 

\begin{thm}\label{maincutnearlylinear}
For each proper hereditary permutation property $\mathcal{P}$ and $\epsilon>0$, let $C=C(\mathcal{P})=1000c(\sigma)$, where $c(\sigma)$ is the F\"uredi-Hajnal constant of a smallest forbidden subpermutation $\sigma$ for $\mathcal{P}$. 
Let $M = M(\epsilon) =\frac{2000C}{\epsilon} \log^2(\epsilon/2)$ and $n_0=n_0(M, \epsilon) = 32M^5k^*(M)/\epsilon^3$. There is a two-sided tester for $\mathcal{P}$ with respect to the planar tau distance of query complexity $M$ for permutations of size at least $n_0$. 
\end{thm} 
The tester works as follows. Let $M$ be as specified in Theorem \ref{maincutnearlylinear}. For a permutation $\pi$, we pick a subpermutation $\pi'$ of $\pi$ of size $M$ uniformly at random; call it an \emph{$M$-sample}. That is, we pick a subset $S \subset [n]$ of size $M$ uniformly at random, and $\pi'$ is the subpermutation of $\pi$ induced on $S$.  If for all integers $k$, there exists a $k$-blow-up of $\pi' \in \PP$, our algorithm outputs ``$\pi$ is in $\PP$''. If there exists an integer $k$ such that no $k$-blow-up of $\pi'$ is in $\PP$, our algorithm outputs ``$\pi$ is not in $\PP$''. We remark that the constant dependence can sometimes be improved by using extremal properties of the family of forbidden subpermutations rather than just the smallest forbidden subpermutation. 

The next theorem gives a universal quadratic bound (not depending on the property) on the query complexity for testing sufficiently large permutations. 

\begin{thm}\label{maincut}
For each hereditary permutation property $\mathcal{P}$ and $\epsilon>0$, let $M = M(\epsilon) =20000/\epsilon^2$ and $n_0 =n_0(M,\epsilon)= 32 M^5 k^*(M)/\epsilon^3$. There is a two-sided tester for $\mathcal{P}$ with respect to the planar tau distance of query complexity $M$ for permutations of size at least $n_0$. 
\end{thm}

Theorems \ref{maincutnearlylinear} and \ref{maincut} are both with respect to two-sided testing. For one-sided testing, we can still get reasonably good bounds, as stated in Theorem \ref{maincutnearlylinearoneside}, by showing that very likely a permutation has the property that it is close to a blow-up of a random subpermutation $\alpha$ and a somewhat larger random subpermutation $\pi'$ very likely contains a $k^*(|\alpha|)$-blow-up of $\alpha$. These bounds are polynomial in $1/\epsilon$ as long as the blow-up parameter $k^*(M)$ for $\mathcal{P}$  is bounded above by a polynomial in $M$. In particular, for almost all permutations $\sigma$ of length $s$, we get a universal bound for testing $\sigma$-freeness. This is because almost all permutations $\sigma$ of length $s$ have no subinterval of length three whose image is an interval of length three, and hence $k^*$ is at most three for the property of being $\sigma$-free.

\begin{thm}\label{maincutnearlylinearoneside}
Let $\mathcal{P}$ be a hereditary permutation property and $\epsilon>0$. Let $M_1(\epsilon), M_2(\epsilon)$ be the value $M$ in terms of $\epsilon$ in Theorems~\ref{maincutnearlylinear} and \ref{maincut}, respectively.  Let $n_0(M, \epsilon) = 32 M^5 k^*(M)/\epsilon^3$.
Let \[M' =  \min\left(n_0(M_1(\epsilon/2), \epsilon/2), n_0(M_2(\epsilon/2), \epsilon/2)\right).\] 
There is a one-sided tester for $\mathcal{P}$ with respect to the planar tau distance of query complexity $M'$ for permutations of size at least $M'$. 
\end{thm} 



Finally, we show that several different permutation metrics of interest are closely related to the cut metric, yielding similar results for testing with respect to these metrics.

We often consider the input permutation $\pi$ of length $n$ as a collection of points $(\frac{i}{n},\frac{\pi(i)}{n})$ in the unit square $[0,1]^2:=[0,1]\times [0,1]$. See Figure \ref{fig:p_representation} for an example. 

\begin{figure}[H]

\centering

\caption{Permutation Representation in $[0,1]^2$. \footnotesize{This is an example of the permutation $8,2,7,6,4,5,3,1,9,10$.}}

\includegraphics[scale=0.3]{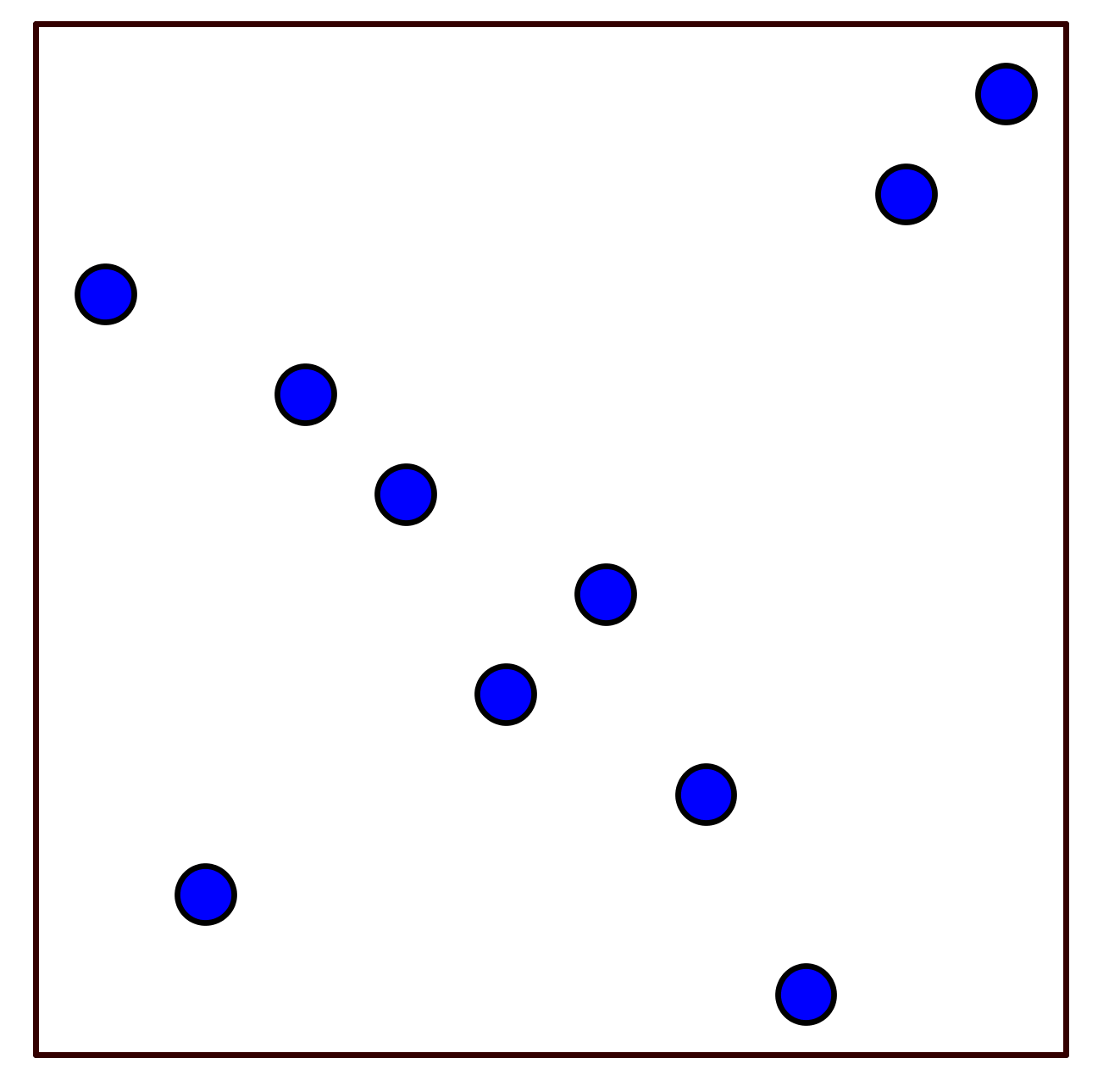}
\label{fig:p_representation}
\end{figure}

If we are testing for a hereditary permutation property $\mathcal{P}$, and we know that the following hold: 

\begin{itemize} 
\item there is a permutation $\sigma \not \in \mathcal{P}$ (or even if large blow-ups of $\sigma$ are not in $\mathcal{P}$), 
\item there are $|\sigma|$ rectangles whose horizontal intervals are disjoint and whose vertical intervals are disjoint, 
\item each of the rectangles contains a significant fraction of the points corresponding to $\pi$, and
\item if we pick one point from each of the rectangles, then we form a copy of $\sigma$, 
\end{itemize} 
then a large sample of points from $\pi$ will likely have many points in each rectangle and thus contain a large blow-up of $\sigma$ and certify that $\pi \not \in \mathcal{P}$. This simple idea is very important for our various property testing algorithms, and is demonstrated in Figure \ref{fig:largepermblowup}.
\begin{figure}[H]

\centering

\caption{Illustration of a permutation containing a large blow-up. \footnotesize{The black dots denote the points in $\pi$; while the red crosses are the points being picked by the $M$-sample.}}

\includegraphics[scale=0.9]{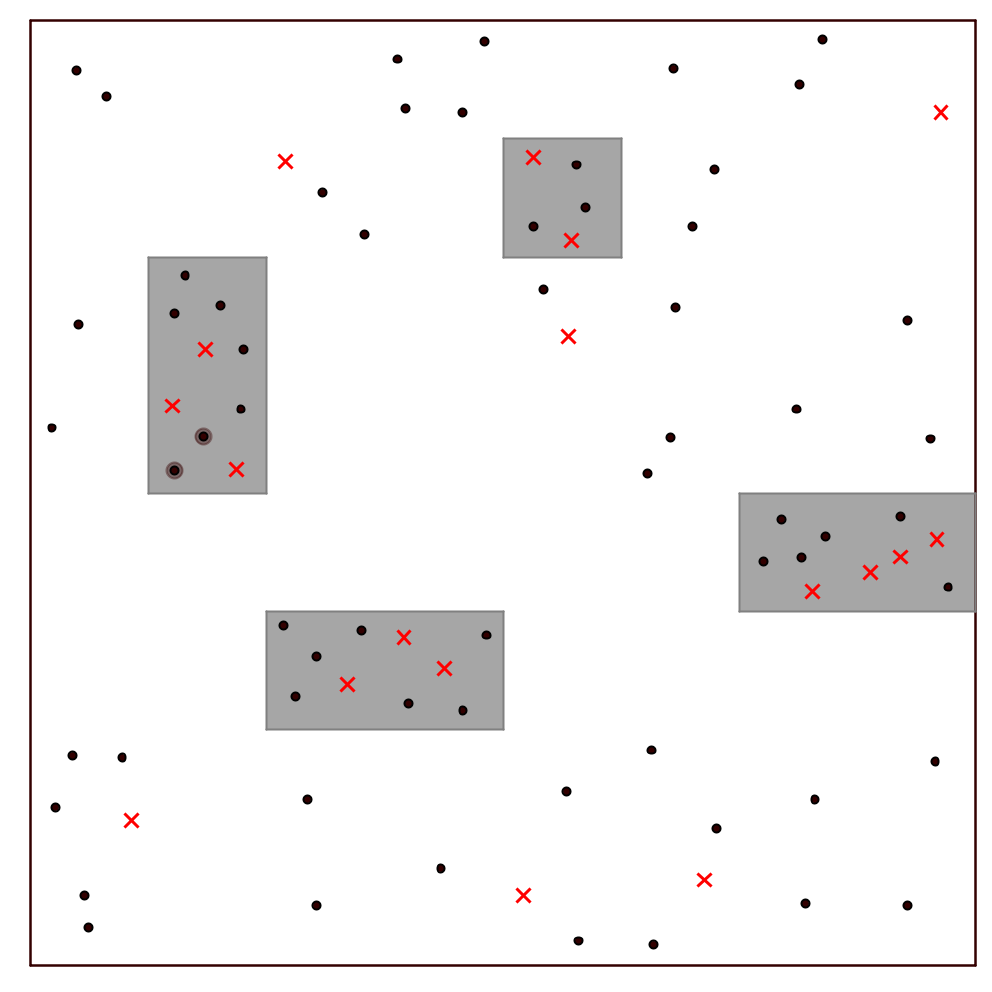}
\label{fig:largepermblowup} 
\end{figure}

\section{Equivalence between different metrics on permutations}
We next define several different metrics between permutations of the same length. Each of the metrics we use here is normalized such that the maximum distance between two permutations cannot exceed 1. We then study properties of these metrics and the relationships between them. Each of these metrics has the property that if two permutations have small distance in one metric, they also have small distance in the other metrics. Furthermore, two permutations have small distance in any of these metrics is equivalent to having, for each small permutation $\mu$, roughly the same density of $\mu$ as a subpermutation. 

Recall that we sometimes view the permutation $\pi$ of length $n$ as the collection of points $(i,\pi(i))$ with $i \in [n]$ or,  normalized, as the collection of points $(\frac{i}{n},\frac{\pi(i)}{n}) \in [0,1]^2$. It should be clear from the context which is used. Viewing a permutation as a collection of points in $[0,1]^2$, given any rectangle $S \subset [0,1]^2$, let $S(\pi)$ be the number of points of $\pi$ inside $S$ (including the boundary).
 
\begin{dfn}
Let $\pi_1, \pi_2$ be permutations of length $n$. 
\begin{enumerate}
\item \emph{Rectangular Distance}, or \emph{Cut Distance}. The rectangular (cut) distance between $\pi_1, \pi_2$ is defined as 
\[  \pc(\pi_1, \pi_2) = \frac{1}{n}\, \, \max_{I, J \text{ intervals } \subset [0,1]} | (I \times J)(\pi_1) -  (I \times J)(\pi_2)|\]
 where the maximum is over all closed intervals $I \subset [0,1]$ and $J \subset [0,1]$. 
\item \emph{Dyadic Distance}. 
A closed interval $I \subset [0,1]$ is called a \emph{dyadic interval} if there exist positive integers $i, k$ such that $I = [\frac{i}{2^k}, \frac{i+1}{2^k}]$.  
The dyadic distance between $\pi_1, \pi_2$ is defined as 
\[ \pd(\pi_1, \pi_2) =  \frac{1}{n} \, \, \max_{I, J \text{ dyadic intervals}} | (I \times J)(\pi_1) -  (I \times J)(\pi_2)|\] where the maximum is over all intervals $I \subset [0,1]$ and $J \subset [0,1]$ and $I, J$ are dyadic intervals. 
\item \emph{Square Distance.} The square distance between $\pi_1, \pi_2$ is defined as 
\[ \ps(\pi_1, \pi_2) = \frac{1}{n} \, \max_{I, J} | (I \times J)(\pi_1) -  (I \times J)(\pi_2)|\]
 where the maximum is over all intervals $I \subset [0,1]$ and $J \subset [0,1]$ and $|I| = |J|$. Thus $I \times J$ is a square in $[0,1]^2$. 
\item \emph{Dyadic Square Distance.} The dyadic square distance between $\pi_1, \pi_2$ is defined as 
\[\pds(\pi_1, \pi_2) = \frac{1}{n} \, \, \max_{I, J \text{ dyadic intervals}, |I| = |J|} | (I \times J)(\pi_1) -  (I \times J)(\pi_2)|\] where the maximum is over all \emph{dyadic squares} $I \times J$, which means $I, J$ are dyadic intervals and $|I| = |J|$. 
\item \emph{Earth Mover's Distance}. The earth mover's distance between $\pi_1, \pi_2$ is defined as 
\[ \pn(\pi_1, \pi_2) =\frac{1}{{n \choose 2}} \, \, \min_{n-\textrm{permutation}~\theta} \left(\sum_{i=1}^n|i-\theta(i)| + |\pi_1(i) - \pi_2(\theta(i))|\right),\]
 where the minimum is over all permutations $\theta: [n] \to [n]$ of length $n$. This is the sum of $L_1$ distances between a point in $\pi_1$ and the point in $\pi_2$ that it maps to under $\theta$. 
\item \emph{Planar tau Distance}. A permutation $\sigma$ can be obtained from a permutation $\pi$ of the same length by a \emph{planar simple transposition} if there exists an integer $1 \leq i < |\pi|-1$ such that $\sigma$ is the same as $\pi$ except either $\sigma(i) = \pi(i+1)$, $\sigma(i+1) = \pi(i)$, or $\sigma^{-1}(i) = \pi^{-1}(i+1)$, $\sigma^{-1}(i+1) = \pi^{-1}(i)$. 
The planar tau distance between $\pi_1, \pi_2$, denoted as $\ppt(\pi_1, \pi_2)$, is the minimum number of planar simple transpositions required to transform $\pi_1$ into $\pi_2$, and then normalized by $1/{n \choose 2}$. Restated, the planar tau distance between $\pi_1$ and $\pi_2$ is the normalized (so divided by $\frac{1}{{n \choose 2}}$)  minimum number of consecutive row or column swaps needed to obtain the permutation matrix of $\pi_2$ from the permutation matrix of $\pi_1$. Recall that the permutation matrix of a permutation $\pi$ of length $n$ is a $n \times n$ matrix $M_{\pi}$ with $M_{i\pi(i)}=1$ for each $i \in [n]$, and the remaining entries are $0$. 
\end{enumerate}
\end{dfn}

We first make several remarks about these metrics. For two permutations of length $n$, their planar tau distance is at most their Kendall's tau distance as Kendall's tau distance is defined in the same way but is more restrictive on the allowed moves (only allowing consecutive column swaps). Similarly, their earth mover's distance is at most their Spearman's tau distance as taking $\theta$ to be the identity permutation of length $n$, we obtain Spearman's footrule distance. Summarizing, the planar distances are at most their classical variants. We also recall Corollary \ref{planarfactortwo}, which shows that $\ppt(\pi_1, \pi_2) \leq \pn(\pi_1, \pi_2) \leq 2\ppt(\pi_1, \pi_2)$. Thus the planar tau distance and earth mover's distance are within a factor two of each other. 

Also, it is worth discussing the complexity of computing these metrics for two permutations of length $n$. The rectangular distance is defined as the minimum over $O(n^4)$ choices of pairs of intervals, while the Square distance is only over $O(n^3)$ choices of pairs of intervals, and the Dyadic distance and the dyadic square distance is defined only over $O(n^2)$ choices of pairs of intervals. Hence, the dyadic distances appear to be considerably faster to determine exactly. 

If we only want to approximate these distances, we can do a much faster computation.  Two squares whose horizontal and vertical intervals differ in endpoints by at most $\epsilon/8$ in each coordinate differ by at most $\epsilon/2$ in the fraction of points of the permutation in the rectangle for a given permutation. Hence, by considering only multiples of $\epsilon/8$ as possible endpoints, we can approximate the rectangular distance within $\epsilon$ using only $O(\epsilon^{-4})$ rectangles. Similarly, we can approximate within $\epsilon$ the square distance by using at most $O(\epsilon^{-3})$ squares, and the dyadic distance or the dyadic square distance within $\epsilon$ using only $O(\epsilon^{-2})$ dyadic rectangles. Thus, these distances can be determined or approximated rather quickly, with the dyadic distances being the fastest to approximate. 

On the other hand, while very natural, the earth mover's distance is defined as the minimum over $n!$ permutations, which requires a huge computation. Similarly, it is unclear if there is an efficient algorithm for computing the planar tau distance efficiently. 
However, it is possible to efficiently approximate these planar distances. By partitioning into boxes of side length about $\epsilon/4$,  just using the information about the fraction of points in each box, by considering roughly the fraction of points in each box that match up to points in other boxes, it is possible to show that one can compute the planar tau distance and the earth mover's distance each within $\epsilon$ in time which is a function only of $\epsilon$.

We next prove Theorem \ref{thmquadraticmetric}, which states that for any two permutations $\pi_1, \pi_2$ of length $n$, we have 
\[  \frac{1}{8} \pc(\pi_1, \pi_2)^2 \leq \pn(\pi_1, \pi_2)  \leq 48 \pc(\pi_1, \pi_2)^{1/2}.  \]

\begin{proof}[Proof of Theorem \ref{thmquadraticmetric}]
The earth mover's distance between two permutations $\pi_1, \pi_2$ is the normalized minimum, over all permutations $\theta: [n] \to [n]$, which is the same as the normalized minimum over all matchings between points in $\pi_1$ and $\pi_2$, of the sum over all matched pairs of the taxicab ($L_1$) distance between the two points that are matched. 

We first show that if $\pc(\pi_1, \pi_2) \geq \epsilon$, then $\pn(\pi_1, \pi_2) \geq \epsilon^2 / 8$. 
Since $\pc(\pi_1, \pi_2) \geq \epsilon$, we can find a rectangle $I \times J \subset [0,1]^2$ such that there are at least $\epsilon n$ more points  in this rectangle in $\pi_1$ than $\pi_2$, or vice versa. Without loss of generality, we assume there are $\epsilon n$ more points in $\pi_1$ than in $\pi_2$. Therefore for any bijection $\theta$ that maps points in $\pi_1$ to the ones in $\pi_2$, at least $\epsilon n$ points in $\pi_1$ have to map to the points of $\pi_2$ that are outside the rectangle $I \times J$.

Assume $I = [x_1, x_2], J  = [y_1, y_2]$. Let $I' = [\max(x_1 - \epsilon/8,0), \min(x_2+ \epsilon/8,1)]$, and $J' =  [\max(y_1 - \epsilon/8,0), \min(y_2+ \epsilon/8,1)]$. Thus $I \times J \subset I' \times J'$, and the difference between the two rectangles have margin at most $\epsilon/8$. 
Figure \ref{fig:cutnew1} illustrates these two rectangles. 
Since $\pi_2$ is a permutation, there are at most $4 \cdot \epsilon n/8  = \epsilon n/2$ points of $\pi_2$ that are inside the region $(I'\times J') \setminus (I \times J)$. Therefore there are at least $\epsilon n - \epsilon n/2 = \epsilon n/2$ points of $\pi_1$ inside $I \times J$ that have to map to points of $\pi_2$ outside $I' \times J'$. However for these points of $\pi_1$, the distance between it and the point in $\pi_2$ it maps to is at least $\epsilon/8$. Therefore $\pn(\pi_1, \pi_2)  \geq  \epsilon/8 \cdot \epsilon n/2 \cdot \frac{n}{{n \choose 2}} \geq \epsilon^2 / 8$. Thus we proved $\frac{1}{8} \pc(\pi_1, \pi_2)^2 \leq \pn(\pi_1, \pi_2)$.

\begin{figure}[h]
\centering
\caption{Example of the matching process. \footnotesize{The inner rectangle is $I \times J$ while the outer rectangle is $I' \times J'$. The round dots are the points of $\pi_1$ and the crosses are the points of $\pi_2$. If a point in $\pi_1$ is connected to a point of $\pi_2$ by a dashed line, it means they are matched. Since there are many more points of $\pi_1$ than of $\pi_2$ in $I \times J$ and not many points of $\pi_2$ in $(I' \times J') \setminus (I \times J)$, many left-over points of $\pi_1$ in $I \times J$ have to map to points of $\pi_2$ outside $I' \times J'$; each such match has a distance at least $\epsilon/8$ between its points.  }}
\includegraphics[scale=0.7, trim={2cm 2.7cm 7.4cm 3.8cm},clip]{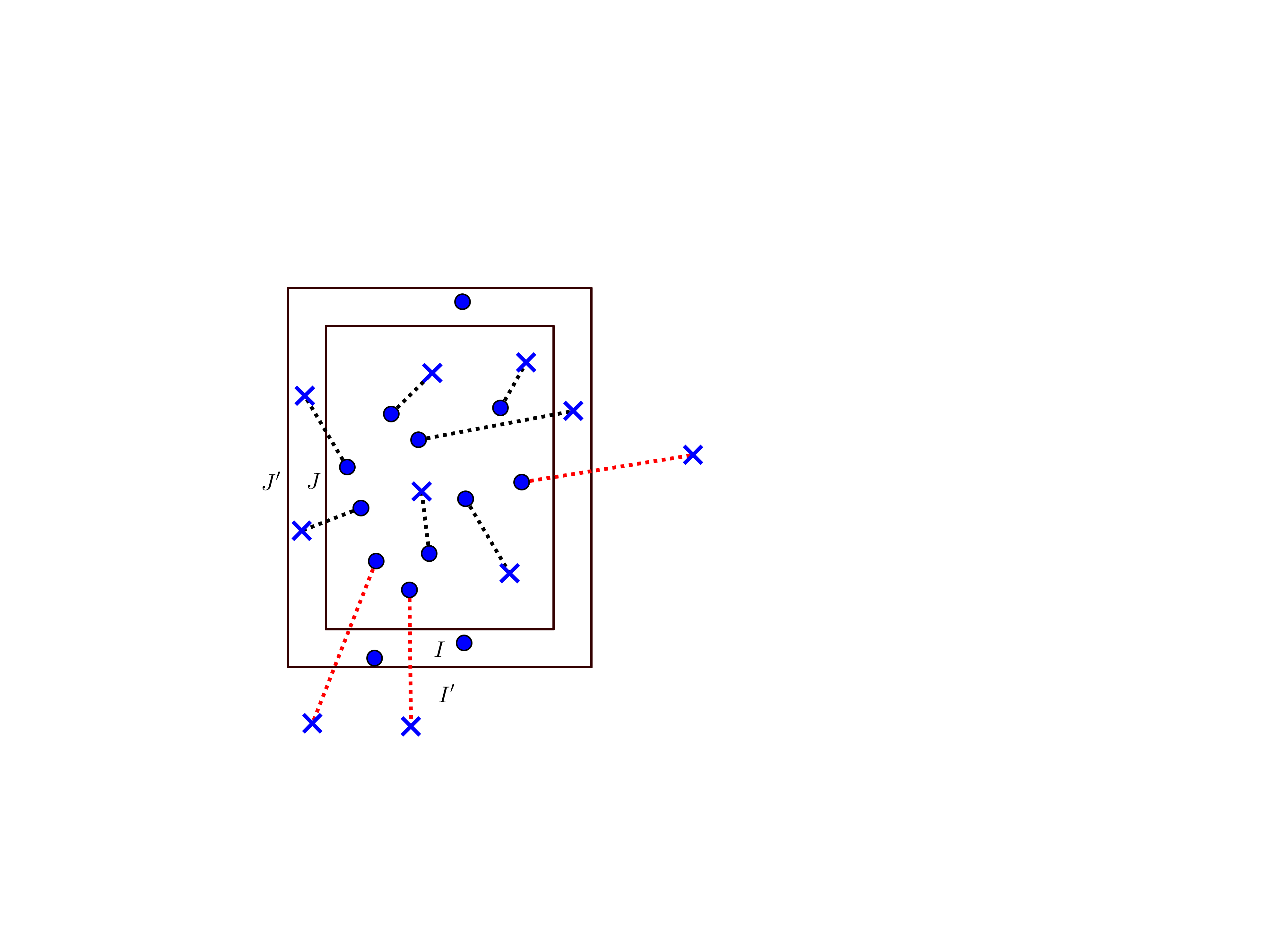}
\label{fig:cutnew1}
\end{figure}

We now show that if $\pc(\pi_1, \pi_2) \leq \epsilon$, then 
$\pn(\pi_1, \pi_2) \leq 48 \sqrt{\epsilon}$. We may assume that $n \geq 2$ as otherwise these distances are all $0$. We will find a bijection $\theta$ such that $\frac{1}{{n \choose 2}}\sum_{i=1}^n|i-\theta(i)| + |\pi_1(i) - \pi_2(\theta(i))| \leq 48\sqrt{\epsilon}$. 
To see this, we partition $[0,1]^2$ into $d^2$ squares, each of side length $1/d$, where $d = 2^{-h}$ with $h=\lfloor \log_2(1/2\sqrt{\epsilon}) \rfloor.$ 
We will define $\theta$ which matches points of $\pi_1$ to points in $\pi_2$ recursively in rounds. In round $0$, we match up as many points of $\pi_1$ to points in $\pi_2$ as possible that lie in the same dyadic square of side length $d$. In round $\ell$, we match up as many not yet matched points of $\pi_1$ to not yet matched points in $\pi_2$ as possible that lie in the same dyadic square of side length $2^{\ell}d$. For each pair of points matched in level $\ell$, their $L_1$ distance is at most $2^{\ell+1}d$. In the last round $\ell=h$, the remaining unmatched points in $\pi_1$ necessarily get matched to the unmatched points in $\pi_2$ as they all lie in the square of side length $1=2^hd$. 

As the discrepancy in the number of points in $\pi_1$ and in $\pi_2$ in any square is at most $\epsilon n$, after round $\ell$, the number of unmatched points is at most $\epsilon n(2^{-\ell}/d)^2$. Thus round $\ell+1$ matches at most $\epsilon n(2^{-\ell}/d)^2$ pairs of points, each such pair of points has $L_1$ distance at most $2^{\ell+2}d$. Hence, the sum of the $L_1$ distances of the pairs matched at level $\ell+1$ is at most $\epsilon n(2^{-\ell}/d)^2 \cdot 2^{\ell+2}d = (2^{2-\ell}/d)\epsilon n$. Summing over all $\ell \geq 0$ gives a total sum of $L_1$ distances in $[0,1]^2$ less than $8\epsilon d^{-1} n$. Also, there are at most $n$ points in $\pi_1$ that get matched to a point in $\pi_2$ in the same dyadic square of side length $d$ so that there $L_1$ distance is at most $2d$. The sum of these $L_1$ distances is at most $2dn$. Thus, the earth mover's distance between $\pi_1$ and $\pi_2$ is at most $$\frac{n}{\binom{n}{2}}\left(2dn+8\epsilon d^{-1} n\right) \leq 8\left(d+4\epsilon d^{-1}\right) \leq 48\sqrt{\epsilon},$$
where the last inequality uses $\epsilon^{1/2} \leq d \leq 2\epsilon^{1/2}$. 
Thus we have proved $\pn(\pi_1, \pi_2) \leq 48\pc(\pi_1, \pi_2)^{1/2}$. 
\end{proof}

The next lemma shows that, up to two logarithmic factors, the rectangular distance is the same as the dyadic distance, which is much faster to compute or approximate.  It is easy to give a construction showing that the upper bound is sometimes tight, and reverse engineering the proof gives a construction showing that the lower bound is sometimes tight up to an absolute constant factor. 

\begin{lem}
\[ \frac{\pc(\pi_1, \pi_2)}{8\log_2^2\left( {8/\pc(\pi_1, \pi_2)}\right)} \leq    \pd(\pi_1, \pi_2) \leq \pc(\pi_1, \pi_2),\]
\end{lem}
\begin{proof}
It is clear from the definition that $\pc(\pi_1, \pi_2) \geq \pd(\pi_1, \pi_2)$.

We thus need to show the other inequality. 
Given $\pc(\pi_1, \pi_2) \geq \epsilon$, we will show $\pd(\pi_1, \pi_2) \geq \frac{\epsilon}{2\log^2 (1/\epsilon)}$.
Since $\pc(\pi_1, \pi_2)\geq \epsilon$, by definition, there exist intervals $I, J\subset [0,1]$ such that $\frac{1}{n}\max_{I, J} | (I \times J)(\pi_1) -  (I \times J)(\pi_2)| \geq \epsilon$. We want to tile most of $I \times J$ by dyadic rectangles; and thus one of them has large difference between the number of points in $\pi_1$ and $\pi_2$. 
We do this by first covering most of the interval $I$ by dyadic intervals, shown in Figure \ref{partI}. 

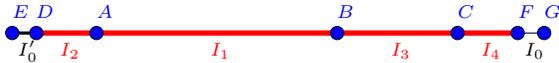
\begin{figure}[h]
\definecolor{ffqqqq}{rgb}{1.,0.,0.}
\definecolor{qqqqff}{rgb}{0.,0.,1.}
\begin{tikzpicture}[line cap=round,line join=round,>=triangle 45,x=0.8cm,y=0.8cm]
\clip(-2.3,2.5) rectangle (7.7,3.7);
\draw [line width=2.pt,color=ffqqqq] (0.,3.)-- (4.,3.);
\draw [line width=2.pt,color=ffqqqq] (4.,3.)-- (6.,3.);
\draw [line width=2.pt,color=ffqqqq] (0.,3.)-- (-1.,3.);
\draw [line width=1.2pt] (-1.,3.)-- (-1.4,3.);
\draw [line width=2.pt,color=ffqqqq] (6.,3.)-- (7.,3.);
\draw (7.,3.)-- (7.44,3.);
\begin{scriptsize}
\draw [fill=qqqqff] (0.,3.) circle (2.5pt);
\draw[color=qqqqff] (0.14,3.36) node {$A$};
\draw [fill=qqqqff] (4.,3.) circle (2.5pt);
\draw[color=qqqqff] (4.14,3.36) node {$B$};
\draw[color=ffqqqq] (2.06,2.7) node {$I_1$};
\draw [fill=qqqqff] (6.,3.) circle (2.5pt);
\draw[color=qqqqff] (6.14,3.36) node {$C$};
\draw[color=ffqqqq] (5.06,2.7) node {$I_3$};
\draw [fill=qqqqff] (-1.,3.) circle (2.5pt);
\draw[color=qqqqff] (-0.86,3.36) node {$D$};
\draw[color=ffqqqq] (-0.44,2.7) node {$I_2$};
\draw [fill=qqqqff] (-1.4,3.) circle (2.5pt);
\draw[color=qqqqff] (-1.26,3.36) node {$E$};
\draw[color=black] (-1.14,2.7) node {$I_0'$};
\draw [fill=qqqqff] (7.,3.) circle (2.5pt);
\draw[color=qqqqff] (7.14,3.36) node {$F$};
\draw[color=ffqqqq] (6.56,2.7) node {$I_4$};
\draw [fill=qqqqff] (7.44,3.) circle (2.5pt);
\draw[color=qqqqff] (7.58,3.36) node {$G$};
\draw[color=black] (7.28,2.7) node {$I_0$};
\end{scriptsize}
\end{tikzpicture}
\caption{Illustration of Partitioning most of an interval $I$ into dyadic intervals. \footnotesize{lnterval $I$ is the segment $EG$. It is mostly covered by four dyadic intervals $I_1 = AB, I_2 = DA, I_3 = BC , I_4 = CF$. The only parts in $I$ not covered by dyadic intervals are $I_0, I_0'$.}}
\label{partI}
\end{figure}

We know $I$ must contain a dyadic interval $I_1$ such that $I_1 \subset I$ and $|I_1| \geq \frac{1}{2} |I|$. Removing $I_1$ from $I$, we are left with at most two other intervals $I_2'$ and $I_3'$ with $I_2'$ coming before $I_3'$ and each of length at most $|I|/2$. In the next level, notice that since $I_1$ is a dyadic interval, $I_2', I_3'$ each have one endpoint being dyadic (i.e., of the form $i/2^k$ for some integers $i, k$). Thus again we can find a dyadic interval $I_2\subset I_2'$ with the same right endpoint as $I_2'$ and $|I_2| \geq |I_2'|/2$; again $I_2' \setminus I_2$ is another interval with one endpoint dyadic. We can similarly find a dyadic interval $I_3 \subset I_3'$ with the same left endpoint as $I_3'$ with $|I_3| \geq |I_3'|/2$ and such that $I_3' \setminus I_3$ is an interval with one endpoint having a dyadic coordinate. We know $ |I \setminus (I_1 \cup I_2 \cup I_3)| \leq (|I \setminus I_1| + \frac{|I_1|}{2}) \leq \frac{|I|}{4}$. Therefore removing $I_1, I_2, I_3$ from $I$ leaves us with at most two remaining intervals with each having a dyadic endpoint and their total length is at most $|I|/4$. We repeat this process. In each step, we find at most two new dyadic intervals and removing them further from $I$ leaves us with two remaining intervals each with one dyadic endpoint and each of these remaining intervals has length at most half of the length of the intervals they came from in the previous step. Thus, we can partition $I$ into at most $1+2\log(4/\epsilon) = 5+2\log(1/\epsilon)$ dyadic intervals $I_1,I_2,\ldots$ and at most two intervals $I_0,I_0'$ such that $|I_0|+|I'_0| \leq \epsilon/4$. This is because in the first step we used one dyadic interval $I_1$, and in each further step, we picked out two dyadic intervals, and the total number of steps used is at most $1+\log_2(\frac{|I|/2}{\epsilon/8}) \leq 1+\log(4/\epsilon)$.  Similarly, we can partition $J$ into at most $1+2\log(4/\epsilon) = 5+2\log(1/\epsilon)$ dyadic intervals $J_1,J_2,\ldots$ and at most two intervals $J_0,J_0'$ such that $|J_0|+|J'_0| \leq \epsilon/4$. 


Therefore we can cover most of the rectangle $I\times J$ by at most $(5 + 2 \cdot \log_2 \frac{1}{\epsilon})^2$ dyadic rectangles $I_i \times J_j$ except  for four rectangles $S_1 = I_0 \times J, S_2 = I_0' \times J, S_3 = I \times J_0, S_4 = I \times J_0'$. This is illustrated in Figure \ref{partr}. 

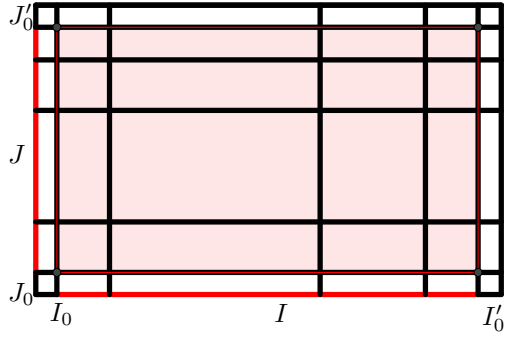
\begin{figure}
\center
\definecolor{uuuuuu}{rgb}{0.26666666666666666,0.26666666666666666,0.26666666666666666}
\definecolor{ffqqqq}{rgb}{1.,0.,0.}
\begin{tikzpicture}[line cap=round,line join=round,>=triangle 45,x=0.7cm,y=0.7cm]
\clip(-4.238,2.2) rectangle (10.106,8.6);
\fill[color=ffqqqq,fill=ffqqqq,fill opacity=0.1] (-1.,8.08) -- (-1.,3.42) -- (7.,3.42) -- (7.,8.08) -- cycle;
\draw [line width=2pt,color=ffqqqq] (0.,3.)-- (4.,3.);
\draw [line width=2pt,color=ffqqqq] (4.,3.)-- (6.,3.);
\draw [line width=2pt,color=ffqqqq] (0.,3.)-- (-1.,3.);
\draw [line width=2pt] (-1.,3.)-- (-1.4,3.);
\draw [line width=2pt,color=ffqqqq] (6.,3.)-- (7.,3.);
\draw [line width=2pt] (7.,3.)-- (7.44,3.);
\draw [line width=2pt] (-1.4,3.)-- (-1.4,3.42);
\draw [line width=2pt,color=ffqqqq] (-1.4,3.42)-- (-1.4,4.38);
\draw [line width=2pt,color=ffqqqq] (-1.4,4.38)-- (-1.4,6.5);
\draw [line width=2pt,color=ffqqqq] (-1.4,6.5)-- (-1.4,7.46);
\draw [line width=2pt,color=ffqqqq] (-1.4,7.46)-- (-1.4,8.08);
\draw [line width=2pt] (-1.4,8.08)-- (-1.4,8.5);
\draw [line width=2pt] (-1.4,8.5)-- (7.44,8.5);
\draw [line width=2pt] (-1.4,8.5)-- (7.44,8.5);
\draw [line width=2pt] (7.44,8.5)-- (7.44,3.);
\draw [line width=2pt] (-1.4,8.08)-- (7.44,8.08);
\draw [line width=2pt] (-1.4,7.46)-- (7.44,7.46);
\draw [line width=2pt] (-1.4,6.5)-- (7.44,6.5);
\draw [line width=2pt] (-1.4,4.38)-- (7.44,4.38);
\draw [line width=2pt] (-1.4,3.42)-- (7.44,3.42);
\draw [line width=2pt] (-1.,3.)-- (-1.,8.5);
\draw [line width=2pt] (0.,3.)-- (0.,8.5);
\draw [line width=2pt] (4.,3.)-- (4.,8.5);
\draw [line width=2pt] (6.,3.)-- (6.,8.5);
\draw [line width=2pt] (7.,3.)-- (7.,8.5);
\draw [color=ffqqqq] (-1.,8.08)-- (-1.,3.42);
\draw [color=ffqqqq] (-1.,3.42)-- (7.,3.42);
\draw [color=ffqqqq] (7.,3.42)-- (7.,8.08);
\draw [color=ffqqqq] (7.,8.08)-- (-1.,8.08);
\draw (2.956,3) node[anchor=north west] {$I$};
\draw (-2.1,6.014) node[anchor=north west] {$J$};
\draw (-1.31,3.022) node[anchor=north west] {$I_0$};
\draw (6.9,2.978) node[anchor=north west] {$I_0'$};
\draw (-2.1,3.4) node[anchor=north west] {$J_0$};
\draw (-2.1,8.7) node[anchor=north west] {$J_0'$};
\begin{scriptsize}
\draw [fill=uuuuuu] (-1.,8.08) circle (1.5pt);
\draw [fill=uuuuuu] (-1.,3.42) circle (1.5pt);
\draw [fill=uuuuuu] (7.,3.42) circle (1.5pt);
\draw [fill=uuuuuu] (7.,8.08) circle (1.5pt);
\end{scriptsize}
\end{tikzpicture}
\caption{Covering most of the rectangle $I \times J$ by dyadic rectangles by the dyadic partitioning of most of $I, J$. \footnotesize{$I \times J$ is the largest rectangle. The sub-interval $\tilde I$ of $I$ covered by dyadic intervals is shown in red. $\tilde I = I \setminus (I_0 \cup I_0')$. Similar for the sub-interval $\tilde J$ of $J$ covered by dyadic intervals. The dyadic rectangles are defined by the dyadic intervals covering $\tilde I, \tilde J$. Thus the shaded rectangle $\tilde I \times \tilde J$ is partitioned by dyadic rectangles.  } }
\label{partr}
\end{figure}

However, notice that for any rectangle $I' \times J'$, we have $0 \leq (I' \times J')(\pi_1) \leq n  \cdot \min(|I'|, |J'|)$ simply because $\pi_1$ is a permutation, and the same applies to $\pi_2$. Therefore $0 \leq (\bigcup_{i=1}^4 S_i)(\pi_1) \leq |I_0| + |I_0'|+|J_0| + |J_0'| \leq  n\epsilon/2$. Similarly, $0 \leq (\bigcup_{i=1}^4 S_i)(\pi_2) \leq n\epsilon/2$. Thus
$| (\bigcup_{i=1}^4 S_i)(\pi_1)  -(\bigcup_{i=1}^4 S_i)(\pi_2) | \leq n\epsilon/2$.
Let $\tilde I = I \setminus (I_0 \cup I_0')$, and $\tilde J = J \setminus (J_0 \cup J_0')$.
Since $|(I\times J)(\pi_1) - (I\times J)(\pi_2)| \geq n\epsilon$ and $\bigcup_{i=1}^4 S_i = (I \times J) \setminus (\tilde I \times \tilde J)$, we have that
 \[ |( \tilde I \times \tilde J)(\pi_1) - (\tilde I\times\tilde  J)(\pi_2)| \geq n\epsilon -n\epsilon/2 \geq n\epsilon /2.\] 
However, we have shown that $\tilde I \times \tilde J$ is covered by at most $(5 + 2 \cdot \log_2 \frac{1}{\epsilon})^2$ dyadic rectangles; therefore there must be a dyadic rectangle $S$ such that 
 \[ |S(\pi_1) - S(\pi_2)| \geq \frac{n\epsilon}{2\left(5 + 2 \cdot \log_2 \frac{1}{\epsilon}\right)^2} \geq \frac{n\epsilon}{8\left(\log_2 \frac{8}{\epsilon}\right)^2} .\] 
This implies 
$
\pd( \pi_1,\pi_2) \geq  \frac{\epsilon}{8\left(\log_2 \frac{8}{\epsilon}\right)^2}.$
\end{proof}

The next lemma relates the rectangular distance and the square distance. It is easy to see that the lower bound is sometimes tight, and reverse engineering the proof shows that the upper bound is sometimes tight up to an absolute constant factor. 
\begin{lem}
\[ \ps( \pi_1,\pi_2) \leq \pc( \pi_1,\pi_2) \leq \sqrt{7} \ps( \pi_1,\pi_2)^{1/2}.\]
\end{lem}
\begin{proof}

It is clear from the definition that $\pc(\pi_1, \pi_2) \geq \ps(\pi_1, \pi_2)$.

We thus need to show the other inequality. 
Given $\pc( \pi_1,\pi_2) \geq \epsilon$, we will show $\ps( \pi_1,\pi_2) \geq \epsilon^2/7$.

As $\pc( \pi_1,\pi_2)\geq \epsilon$, there exist intervals $I, J\subset [0,1]$ such that $\frac{1}{n}\max_{I, J} | (I \times J)(\pi_1) -  (I \times J)(\pi_2)| \geq \epsilon$. We want to find a square in $I \times J$ in which there are many more points from $\pi_1$ than from $\pi_2$, or vice versa. 

Assuming $|I| \geq |J|$, then we partition $I$ into $\lfloor |I|/|J| \rfloor$ intervals of length $|J|$ and a remaining interval with length $|I| - |J| \lfloor |I|/|J| \rfloor$. Thus the rectangle $J \times I$ is partitioned into $\lfloor |I|/|J| \rfloor$ squares each of side length $|J|$ and a remaining rectangle of size $|J| \times (|I| - |J| \lfloor |I|/|J| \rfloor)$. An example of this step can be seen as the three largest squares in Figure \ref{partrs}. And we do the same partitioning procedure for the remaining rectangle by covering the longer side (now it is the interval of length $|J|$) by as many intervals of length $(|I| - |J| \lfloor |I|/|J| \rfloor)$ as possible. We repeat this process, until the remaining interval has length at most $\epsilon/2$. In Figure \ref{partrs}, we repeat this step for four rounds (corresponding to squares of four different sizes), and stop at the smallest white rectangle since its shorter side has length at most $\epsilon/2$.

\begin{figure}[H]

\centering

\caption{Illustration of covering most of a rectangle by squares. \footnotesize{In this example, the rectangle $I \times J$ is covered by squares of four different sizes (the shaded squares) and a remaining rectangle $R$ (the not shaded rectangle in the top right corner) with one side length smaller than $\epsilon/2$.}}
\label{partrs}

\includegraphics[scale=0.5]{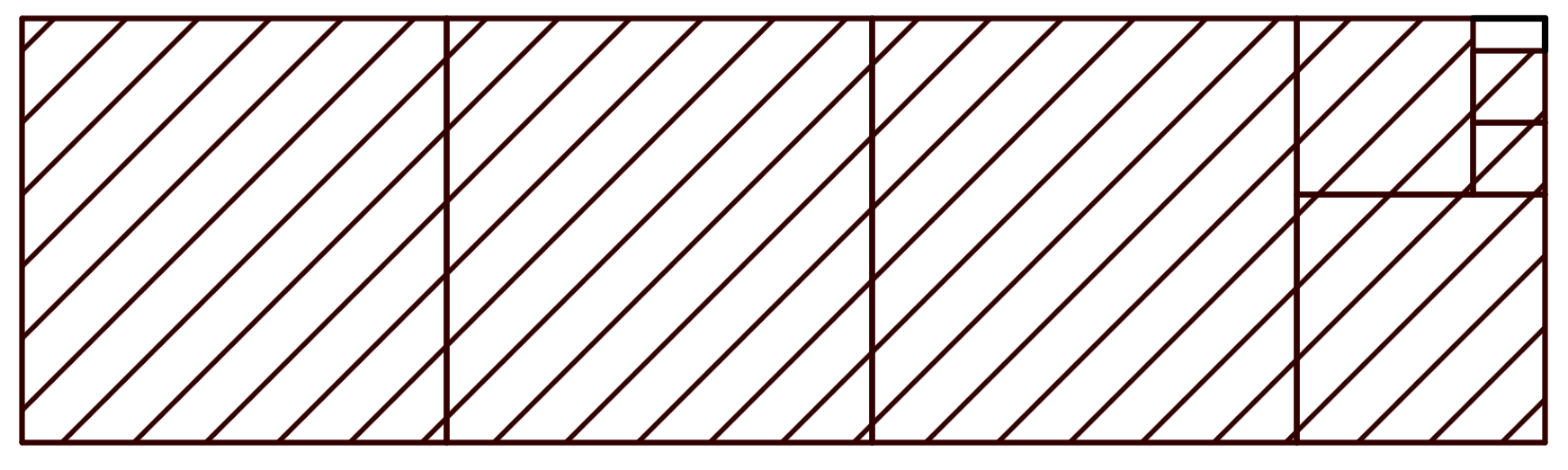}
\end{figure}

We want to bound the number of squares we obtained. 
Notice that for each rectangle of size $a_1 \times a_2$, assuming $a_1 \geq a_2$, after cutting it into as many squares of size $a_2 \times a_2$ as possible, the remaining rectangle has size $a_3 \times a_2$, where $0 \leq a_3 < a_2$ satisfies $a_1 = k_1a_2 + a_3$ for some positive integer $k$. Since $k_1 \geq 1$ and $a_3 < a_2$, clearly we have $a_3 \leq a_1/2$. 
We then partition the rectangle $a_3 \times a_2$ into $\lfloor a_2/a_3 \rfloor$ squares of size $a_3 \times a_3$, and a remaining rectangle is of size $a_3 \times a_4$ where $0 \leq a_4 < a_3$. By the same argument, again we have $a_4 \leq a_2/2$. 
We repeat the process, obtaining squares of side lengths $a_1 > a_2 > \dots$ until the side length of some rectangle is no more than $\epsilon/2$. Say it stops after subdividing a rectangle of size $a_{N-1} \times a_N$ into many squares, and the remaining rectangle is of size $a_N \times a_{N+1}$ where $0 \leq a_{N+1} \leq \epsilon/2$.  
Since $a_{i} \leq a_{i-2}/2$ and $a_{N} \geq \epsilon/2$, we have that 
\beq N \leq 2\log_2 (2/\epsilon). \label{boundN} \eeq 
The total number of squares  is bounded above by 
\beq a_1 / a_2 + a_2/a_3 + \dots + a_{N-1}/a_N. \label{numbersquare} \eeq
Now, to upper bound the sum in (\ref{numbersquare}), notice that the product of the terms is $a_1 / a_N <1 / (\epsilon/2) = 2/\epsilon$ since $a_N > \epsilon/2$. Furthermore, for each $i \leq N-1$, $1 < a_i / a_{i+1} \leq 1/ (\epsilon/2) = 2/\epsilon$ as $a_i \leq 1, a_{i+1} \geq \epsilon/2$ for $i \leq N-1$. Given an upper bound of the product of non-negative terms and the range for each term, we know that the sum of the terms is maximized when the terms are at their extreme values; in our case, it would be that all the terms are either $2/\epsilon$ or $1$. Therefore (\ref{numbersquare}) is bounded above by $(N-2) + 2/\epsilon$. Plugging in (\ref{boundN}), the number of squares is bounded above by $2\log_2 (2/\epsilon) + 2/\epsilon$. 

Thus we have partitioned the rectangle $I \times J$ into at most $2\log_2 (2/\epsilon) + 2/\epsilon$ squares and a rectangle $R$ with one side length at most $\epsilon/2$. Thus $|R(\pi_1) - R(\pi_2)| \leq \epsilon/2$. Therefore 
$\frac{1}{n} | ((I \times J) \setminus R)(\pi_1) -  ((I \times J)\setminus R)(\pi_2)| \geq \epsilon/2$, where $(I \times J) \setminus R$ is partitioned into at most $2\log_2 (2/\epsilon) + 2/\epsilon$ squares. Thus one of the squares, call it $S$, satisfies
$\frac{1}{n} | S(\pi_1) -  S(\pi_2)| \geq \frac{\epsilon}{2(2\log_2 (2/\epsilon) + 2/\epsilon)} \geq \epsilon^2/7.$ The last inequality holds for all $0<\epsilon \leq 1$. 

\end{proof}

\section{Two-sided property testing under the planar metrics}\label{sec:twosidedpt}

Recall from Corollary \ref{planarfactortwo} that the planar tau distance and the earth mover's distance are within a factor two. Thus, to test with respect to either metric is essentially the same thing. Hence, we can pick the metric which is the easiest to analyze, and, for our purposes, it is the earth mover's distance. 

For a hereditary property $\PP$, we propose the following algorithm, a two-sided tester for property $\PP$. It is closely related to the blow-up parameter for $\PP$, which is defined in Definition \ref{def:blowupp}.

\begin{algorithm}
\caption{Two-sided Tester}\label{tester:1}
Let $M$ be as specified in Theorem \ref{newm}. Given a permutation $\pi$ of length $n$, we pick a subpermutation $\pi'$ of $\pi$ of size $M$ uniformly at random. That is, we pick a subset $S \subset [n]$ of size $M$ uniformly at random, and $\pi_0$ is the subpermutation of $\pi$ induced on $S$. 

\begin{enumerate}
\item If for all integers $k$, there exists a $k$-blow-up of $\pi' \in \PP$, our algorithm outputs ``$\pi$ is in $\PP$''. Equivalently, if $k^*(M) = \infty$, then output ``$\pi$ is in $\PP$'', and otherwise $k^*(M)$ is finite (Recall the definition of $k^*(M)$ in Definition \ref{def:blowupp}), and we output ``$\pi$ is in $\PP$"  if and only if there is a $k^*(M)$-blow-up of $\pi' \in \PP$.
\item If no $k^*(M)$-blow-up of $\pi'$ is in $\PP$, our algorithm outputs ``$\pi$ is not in $\PP$''.
\end{enumerate}
\end{algorithm}

We will analyze this algorithm and show its effectiveness in the following theorem. Let $\dist_{\text{PF}}(\pi, \PP)$ be the minimum of $\pn( \pi,\pi_0)$ over all $\pi_0 \in \PP$ of the same size as $\pi$. If $\PP$ has no permutation of size $| \pi |$, then  $\dist_{\text{PF}}(\pi, \PP)$ is defined to be $1$. 

We restate Theorem \ref{maincutnearlylinear} below for convenience. 

\begin{thm}\label{newm}
Let $\PP$ be a proper hereditary family of permutations, $\sigma$ be a smallest permutation not in $\PP$, and $s = |\sigma|$.
Let $c(\sigma)$ be the F\"uredi-Hajnal constant of $\sigma$.

For any fixed $0<\epsilon \leq  1$, let $M = \frac{2000 c(\sigma)\log^2(2/\epsilon)}{\epsilon}$, so $M \geq |\sigma|$ and hence $k^*(M)$ is finite. Then for any permutation $\pi$ of length at least $32M^5 k^*(M) / \epsilon^3$, the following holds. 
\begin{enumerate}
\item If $\pi \in \PP$, then with probability at least $1-\epsilon$, the algorithm outputs ``$\pi$ is in $\PP$". 
\item If $\dist_{\text{PF}}(\pi, \PP) > \epsilon$, then with probability at least $1- \epsilon$, the algorithm outputs ``$\pi$ is not in $\PP$''.
\end{enumerate}
\end{thm}

We justify these two assertions in the next two subsections, respectively. 

\subsection{When $\pi \in \PP$.}
We first show that when $\pi \in \PP$, with large probability the algorithm outputs $\pi \in \PP$.  

For this part, we need the following key lemma.

\begin{lem}\label{typicalsample}
Suppose $|\pi| = n$. Let $T, k$ be positive integers and $\epsilon'$ a positive number. Suppose $n \geq 32T^5k/\epsilon'^3$. Let $\pi'$ be a subpermutation of $\pi$ of size $T$ picked uniformly at random. Then, with probability at least $1-\epsilon'$, some $k$-blow-up of $\pi'$ is in $\pi$. 
\end{lem}
\begin{proof}
We consider the permutation as a collection of points represented in the unit square $[0,1]^2$.  Let $t = \frac{4T^2}{\epsilon'}$. We equally partition $[0,1]^2$ into $t$ columns and $t$ rows, thus each of the $t\times t$ squares is of size $1/t \times 1/t$. We say a square is dense if it contains at least $\delta n $ points of $\pi$, where $\delta = \frac{\epsilon'^3}{32T^5}$. Otherwise we say the square is sparse. 

For the samples of size $T$, the probability that those samples are in different columns is $t(t-1)\dots (t-T+1) / t^T$, where the nominator is the number of ways of assigning the $T$ samples to columns if they are in different columns, and the denominator is an upper bound of the total number of ways which columns these $T$ samples can be in. Therefore, the probability that at least two samples are in the same column is at most
$ 1 - \frac{t(t-1)\dots (t-T+1)}{ t^T}  \leq 1 - \left( \frac{t-T+1}{t} \right)^T \leq \frac{T(T-1)}{t}.$
Similarly, the probability that at least two samples are in the same row is at most  $\frac{T(T-1)}{t}$. Thus, by the union bound, the probability that there are two samples in the same column or row is at most $\frac{2T(T-1)}{t}$.

For a sample point, the probability that it lies in a sparse square is at most $\delta t^2$ since each sparse square has at most $\delta n$ points in $\pi$ and there are at most $t^2$ sparse regions. Thus by the union bound, the probability that at least one of the samples is in the sparse region is at most $\delta  t^2 T$. 

Combining the results above, the probability that the $T$ samples are in different columns and different rows, and all the samples are in dense squares is at least 
\[ 1 - \frac{2T(T-1)}{t} - \delta T t^2 \geq 1- \epsilon'.\] 
The inequality holds by our choice of $\delta$ and $t$.
We thus now assume the $T$ sample points are in different dense squares and they are in different rows and columns. However, in each of the dense squares that the samples are in, there are at least $\delta n \geq k$ points of $\pi$ in that square.  By picking $k$ points of $\pi$ from each of these squares give us a $k$-blow-up of the $T$-sample $\pi'$. Therefore a $k$-blow-up of $\pi'$ is in $\pi$. 
\end{proof}

With Lemma \ref{typicalsample}
we can prove the first assertion of Theorem \ref{newm}. 
\begin{proof}[Proof of the first assertion in Theorem \ref{newm}]
Recall the definition of the blow-up parameter $k^*(\cdot)$ discussed in Definition \ref{def:blowupp}. 

Let $q= k^*(M)$, which is finite. Let $\pi'$ be the $M$-subpermutation of $\pi$ picked uniformly at random. 
Applying Lemma \ref{typicalsample} with $k = q, \epsilon' = \epsilon, T = M$, we know as long as $n \geq 32M^5 q / \epsilon^3$,
 with probability at least $1-\epsilon$, the $M$-sample $\pi'$ has a $q$-blow-up in $\pi$. Since $\pi \in \PP$ and $\PP$ is hereditary, this $q$-blow-up of $\pi'$, which is a subpermutation of $\pi$, is in $\PP$. 
However, if our algorithm outputs ``$\pi$ is not in $\PP$", it means there exists a least  positive integer $\ell$ such that no $\ell$-blow-up of $\pi'$ is in $\PP$. By our definition of $k^*(M)$, we have $l \leq k^*(M) = q$. But we have just found a $q$-blow-up of $\pi'$ in $\PP$, which contradicts $\PP$ is hereditary.
\end{proof}

\subsection{When $\pi$ is $\epsilon$-far from satisfying $\PP$}
We now prove the second part of Theorem \ref{newm}. We want to show that if $\pn(\pi, \pi_0) \geq \epsilon$ for all $\pi_0 \in \PP$, then with large probability, our algorithm will output ``$\pi$ is not in $\PP$".

\paragraph{Proof outline.}
The proof idea is as follows. 

We view the permutation $\pi$ as a collection of points in $[0,1]^2$. We carry out a ``partition procedure," later described as Algorithm \ref{partitionnm}, to partition $[0,1]^2$ into dyadic squares (of possibly different sizes), where some squares are dense (i.e., informally containing many points from $\pi$) and some squares are sparse (i.e., informally containing only few points from $\pi$). With large probability, the $M$-sample will hit each of the dense squares (Lemma \ref{densesample}). 

Roughly speaking, if there are many squares that are dense, then $\pi$ contains many copies of $\sigma$, the smallest permutation not in $\PP$. Thus, with large probability, the $M$-sample will contain a copy of $\sigma$. This is proved in Lemma \ref{manyf}. Hence, our algorithm will output ``$\pi$ is not in $\PP$". 

Therefore, we can restrict our attention to the case where not many squares are dense. In this case, we show that with large probability, the $M$-sample $\pi'$ contains a subpermutation $\alpha$ which has a blow-up $\ta$ that is close to $\pi$. In fact, $\pn(\pi, \ta) < \epsilon/2$ (Lemma \ref{pita}). If our algorithm outputs ``$\pi$ is in $\PP$", then for any $k$, there is some $k$-blow-up of $\pi'$ in $\PP$. Since $\PP$ is hereditary, there is some blow-up of $\alpha$, call it $\ta'$, which has the same block sizes as $\ta$, and $\ta' \in \PP$. But $\pn(\ta, \ta')<\epsilon/2$ (Lemma \ref{tata'}). Thus $\pi$ is within distance $\epsilon$ from $\ta' \in \PP$ (Lemma \ref{closeto}). This contradicts the fact that $\pi$ is $\epsilon$-far from $\PP$ under $\pn$.

We first partition $[0,1]^2$ into squares of different types depending on the number of points of the permutation in the square.
\paragraph{Dyadic Partition Procedure. }
This is done through a dynamic procedure in different levels, starting from level 0. Throughout the procedure, we maintain and update three sets of squares, called $\Active$, $\Mature$, and $\Frozen$, which roughly correspond to squares which are dense, modestly dense, and sparse, respectively. 
From now on, let 
\beq \delta =\frac{\epsilon }{512 c(\sigma) \log(32/\epsilon)}  \label{sec3delta} \eeq
 be the density parameter for a square. Here density means the fraction of the points of $\pi$ that is in this square. (Notice here we do not normalize the number of points of $\pi$ by the area of the square). A square in $\Active$ means it has density at least $\delta$ (thus dense); a square in $\Frozen$ means it has density less than $\delta$ (thus sparse). We will specify $\Mature$ soon. 

Figure \ref{fig:partitionp} illustrates this procedure and should be helpful for the reader in understanding it.

To initialize the procedure, which begins at level 0, we put $[0,1]^2$ into the set $\Active$ since the density is $1 > \delta$. Initially, the sets $\Mature$ and $\Frozen$ are empty. 

In level $i$, for $i \geq 1$, we consider squares at level $i$, working only with squares in $\Active$ which are of side length $1/2^{i-1}$. This procedure will maintain that at the beginning of level $i$, all the squares in $\Active$ are of side length $1/2^{i-1}$; and at the end of level $i$, all squares in $\Active$ are of side length $1/2^i$.  We remove a square $R$ of side length $1/2^{i-1}$ from $\Active$, and subdivide it into four \emph{children squares} $R_1, R_2, R_3, R_4$, each having side length $1/2^i$. We say $R$ is the \emph{parent square} of $R_i, i=1,2,3,4$. For any of those four squares, it is added to $\Active$ if it has density at least $\delta$, and is added to $\Frozen$ if otherwise. Furthermore, if all those four squares are added to $\Frozen$, then we add $R$ itself to $\Mature$. This is roughly saying that $R$ is dense, since it once was in $\Active$; but it is not very dense, since it has density at most $4 \delta$ as all of its four children squares are added to $\Frozen$. After processing all the squares in $\Active$ of side length $1/2^{i-1}$, we move to the next level.

This procedure stops when $\Active$ is empty, or after level $\lceil  \log(64/\epsilon) \rceil$. This means the smallest possible squares in the three sets have side length at least 
\beq d = 2^{-\lceil  \log(64/\epsilon) \rceil}.  \label{sec3d} \eeq
We have $\epsilon/128 < d \leq \epsilon/64$.

In summary, the procedure is to subdivide parent squares which are dense into children squares, and subdivide its children squares further only if some children squares are still dense. We never subdivide squares in $\Frozen$. Eventually we obtain a partition of $[0,1]^2$ by squares in $\Active$ and $\Frozen$. 
Notice that if a square is labeled as frozen, then its parent square is either mature or once active. 

This procedure is also presented in the pseudocode below. 
\begin{algorithm}[H]
\caption{Dyadic Partition Procedure}\label{partitionnm}
\begin{algorithmic}[1]
\BState{\textbf{Input: }The $[0,1]^2$ representation of $\pi$.}
\BState{\textbf{Output: } Three sets of squares $\Active, \Frozen, \Mature$.}
\BState{$\Active = \{ [0,1]^2\}.$}
\BState{$\Mature = \emptyset.$}
\BState{$\Frozen = \emptyset.$}
\Procedure{Partition $[0,1]^2$}{}
\For {Levels $i = 1, 2, \dots, \log(1/d)$} 
\If {$\Active = \emptyset$}
\State{\textbf{Stop}}
\Else
\For {all squares $R$ in $\Active$ with side length $1/2^{i-1}$}
\State{$\Active = \Active \setminus \{R\}$.} 
\State{Split $R$ into four equal sized squares $R_1, \dots, R_4$, each of side length $1/2^i$.}
\If{None of $R_1, R_2, R_3, R_4$ has density at least $\delta$}
\State{$\Frozen = \Frozen \cup \{R_1, R_2, R_3, R_4\}.$}
\State{$\Mature = \Mature \cup\{ R\}$.} 
\Else
\State{$\Active = \Active \cup \{R_i, i = 1,2,3,4 \text{ which has density at least $\delta$}\}$;}
\State{$\Frozen = \Frozen \cup \{R_i, i = 1,2,3,4 \text{ which has density less than $\delta$}\}$;}
\EndIf
\EndFor
\EndIf
\EndFor
\Return{$\Active, \Frozen, \Mature$.}
\EndProcedure
\end{algorithmic}
\end{algorithm}

\begin{figure}[h]
\centering
\caption{Illustration of the Dyadic Partition Procedure. \footnotesize{We provide an example of the first four levels of the partition procedure. The squares in $\Active$ resulted in each level are denoted by the black squares. At each level, the active squares in $\Active$ splits. The squares in $\Frozen$ at each level are denoted by the white squares with label ``F" (stands for ``Frozen"). The squares in $\Mature$ are given by the squares with the parallel lines filled-in and each such square consists of four frozen squares. In this example, $\Mature$ is empty in levels 0,1,2. After finishing the algorithm at level 3, there are two squares in $\Mature$.}}
\includegraphics[scale=0.7, trim={1.5cm 1.6cm 7cm 1.8cm},clip]{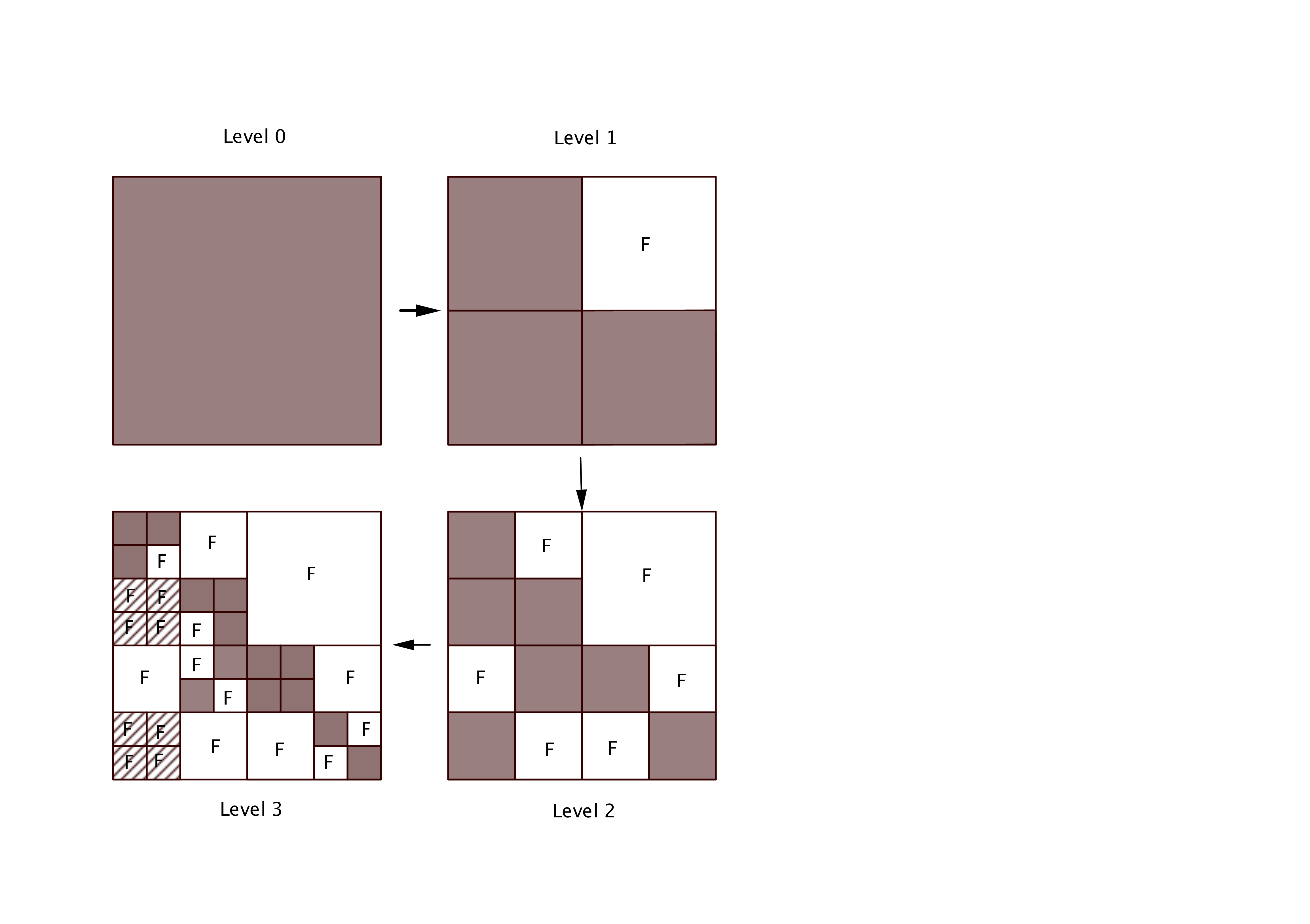}
\label{fig:partitionp}
\end{figure}

Based on this procedure, we obtain the following simple observations. 
\begin{lem}\label{procedurec}
\begin{enumerate}[(a)]
\item The squares in $\Active, \Frozen$ form a partition of $[0,1]^2$.
\item Each square in $\Active$ is disjoint from each square in $\Mature$. 
\item At the end of each level, all squares in $\Active$ have the same side length. If $\Active \neq \emptyset$ when the procedure stops, then all squares in $\Active$ have side length $d$.
\item If $R \in \Frozen$, then the parent square of $R$ has density at least $\delta$. 
\item If $R \in \Mature$, then its density is at most $4\delta$; if $R \in \Active$, then its density at most $d$.
\end{enumerate}
\end{lem}
\begin{proof}
(a), (b), (d) are direct consequences of the procedure.

To see (c), notice that if $R$ is active and has side length greater than $d$, then the procedure has not stopped yet; thus it should have been subdivided.

To see (e), notice that $R$ is in $\Mature$ only if none of its four children squares have density at least $\delta$. Therefore $R$ has density at most $4\delta$. If $R \in \Active$ when the procedure stops, then by (c) it has side length $d$, but since $\pi$ is a permutation, it means $R$ contains at most $dn$ points. 
\end{proof}

Let $\Rich$ be the set of squares at the end of the procedure containing many points, defined as $\Rich = \Active \cup \Mature$. By definition, any square in $\Rich$ has density at least $\delta$. 
In this way we have identified regions in $[0,1]^2$ which contain many points in $\pi$ (the squares in the set $\Rich$). We claim that with large probability, the $M$-sample of $\pi$ hits each of these dense regions.
\begin{lem}\label{densesample}
With probability at least $1- \epsilon$, every square $R \in \Rich$ contains a sample point.
\end{lem}
\begin{proof}
There are at most $(1/d)^2 = d^{-2}$ squares in $\Rich$ since each square has side length at least $d$ and the squares in $\Rich$ are disjoint. 
For each $R \in \Rich$, it contains at least $\delta n$ points of $\pi$. The probability that none of the $M$ samples are in $R$ is bounded above by $(1- \delta)^M$. Thus, by the union bound, the probability that there is a dense square in $\Rich$ which is not hit by any of the $M$ samples is bounded above by 
\[ d^{-2} (1-\delta)^M \leq d^{-2} e^{-\delta M}.\]
However, recall $M = \frac{2000 c(\sigma)\log^2(2/\epsilon)}{\epsilon}$, and the values of $\delta, d$ are determined in (\ref{sec3delta}) and (\ref{sec3d}); by a simple computation,  $d^{-2} e^{-\delta M} < \epsilon$ since $M >- \ln(\epsilon d^2) / \delta$ by our choices of $M, d,\delta$.
\end{proof}

Thus we know with probability at least $1-\epsilon$, each square in $\Rich$ contains a sample point. We pick one sample point from each of the squares in $\Rich$ (if there are multiple samples in  one square, we arbitrarily pick one). These samples of $\pi$ induce a permutation $\alpha$ of length $|\Rich|$. 
For any $1 \leq i \leq |\alpha|$, define $j_i$ to be such that the $i$-th element (from left to right) of $\alpha$ corresponds to the sample $(j_i, \pi(j_i))$ of $\pi$. Thus $1\leq j_1 < \dots < j_{|\alpha|} \leq n$. Furthermore, $\alpha(i) > \alpha(i')$ if and only if $\pi(j_i) > \pi(j_{i'})$. Figure \ref{fig:phi_psi}(a) shows an example of the $M$-sample in $\pi$, and Figure \ref{fig:phi_psi}(b) shows $\alpha$; in this example, $\alpha = 4213$. 

\begin{figure}[h]
\centering
\caption{Obtaining a blow-up of a subpermutation of the $M$-sample that is close to $\pi$. \footnotesize{The black and red points are points in $\pi$. In this example $|\pi| = 20$. 
In Figures \ref{fig:phi_psi}(a),(b),(c), the squares are the results of the partition procedure. The squares in gray are the squares in $\Active$;  The squares in white and the four smaller squares with stripes are in $\Frozen$; The larger striped square, which consists of four smaller striped squares, is in $\Mature$. In this example $|\Active| = 3, |\Frozen| = 13, |\Mature| = 1$. Figure \ref{fig:phi_psi}(a): the red crosses are the $M$ sample points; each region in $\Rich$ contains a sample point. Figure \ref{fig:phi_psi}(b): The permutation $\alpha$ consists of one sample point from each square in $\Rich$, given by the red `+`; if one square in $\Rich$ contains multiple sample points, arbitrarily pick one. We ignore the sample points not contained in squares in $\Rich$.  In this example, the permutation $\alpha$ induced by the subset of sample points is $4213$. Figure \ref{fig:phi_psi}(c): Map $\phi$ maps each point in $\pi$ to a point in $\alpha$. The dashed line connects the point of $\alpha$ that a point of $\pi$ is mapped to. A point is mapped to a point of $\alpha$ sharing the same parent square. Figure \ref{fig:phi_psi}(d): The permutation $\ta$. Notice that in the $t$-th blow-up block, the subpermutation is the one induced by the points in $\pi$ which are mapped to the $t$-th element of $\alpha$. 
}}
\includegraphics[scale=0.9]{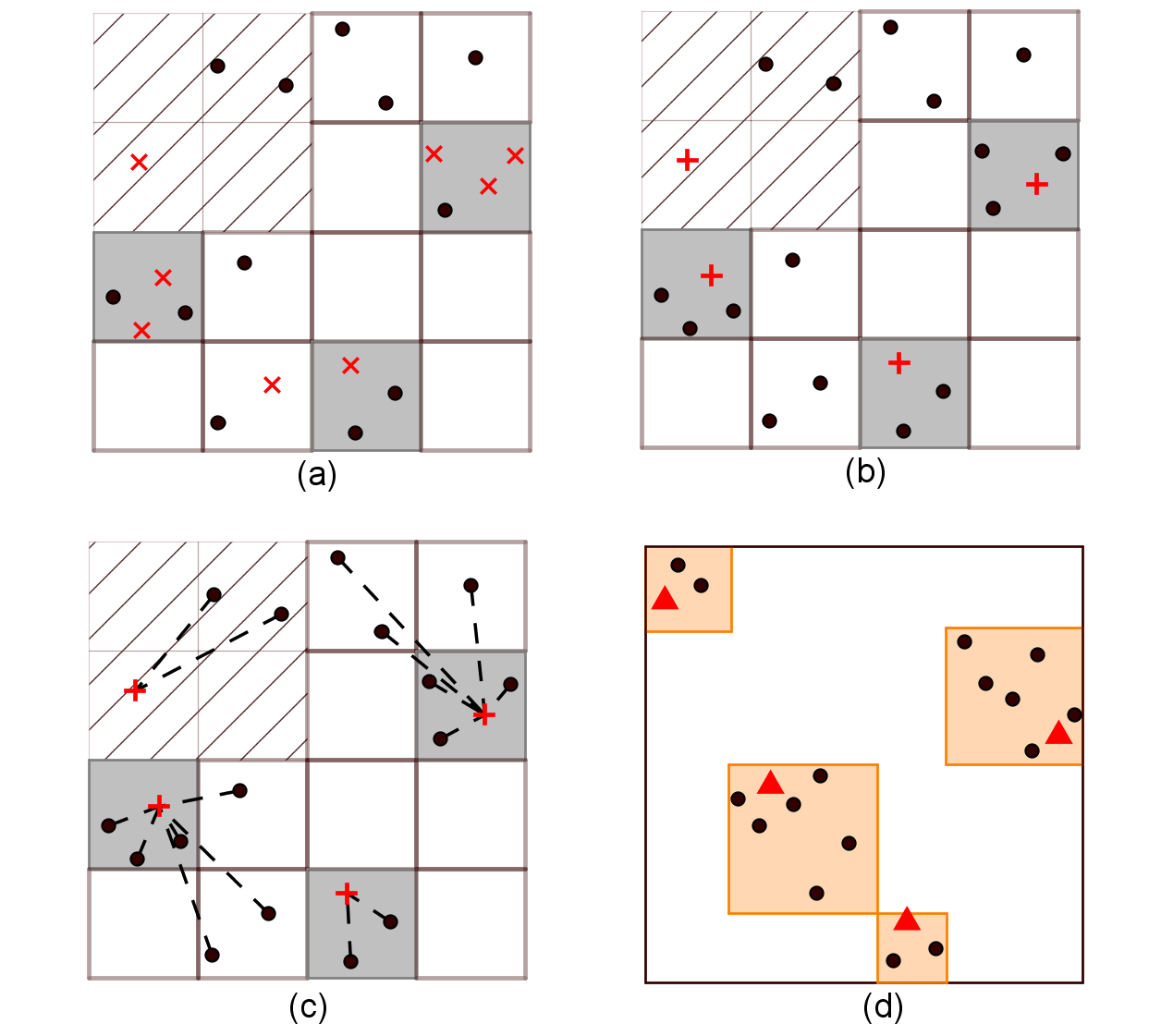}\label{fig:phi_psi}
\end{figure}

By Lemma \ref{procedurec}, the parent of each frozen square was once active or mature. If there are many frozen squares of the same size, then we get many parent squares of the same size which are dense. 
The next lemma shows that if this is the case, our algorithm will output ``$\pi$ is not in $\PP$" with large probability. This lemma implicitly uses the Marcus-Tardos theorem \cite{MT} discussed in the introduction, which verifies the existence of the F\"uredi-Hajnal constant. 
\begin{lem}\label{manyf}
Let $C = 2c(\sigma)$ where $c(\sigma)$ is the F\"uredi-Hajnal constant.
If there is a positive integer $j$ such that there are at least $4C \cdot 2^{j-1}$ squares in $\Frozen$ with side length $1/2^j$, then with probability at least $1 -\epsilon$, our algorithm will output ``$\pi$ is not in $\PP$".
\end{lem}
\begin{proof}
Let the set of squares in $\Frozen$ with side length $1/2^j$ be $\mathcal{F}_j$. For each square in $\mathcal{F}_j$, its parent square is either in $\Mature$ or once was in $\Active$. Consider the set of all these parent squares; call the set $\mathcal{S}$. Thus $|\cF_j| \leq 4|\cS|$ since each parent square has four children squares. Thus $|\cS| \geq C \cdot 2^{j-1}$.

Now we consider $[0,1]^2$ as being partitioned into $(2^{j-1})^2$ regions where each region is a square with side length $1/2^{j-1}$. The parent squares in $\cS$ are some of those regions. Each region corresponding to a parent square in $\cS$ has density at least $\delta$. 
Recall $s = |\sigma|$ and $C =2c(\sigma)$. 
Since $|\cS| \geq C \cdot 2^{j-1}$, by the Marcus-Tardos Theorem, we can find $s$ regions corresponding to $s$ squares in $\cS$ that are in different columns and rows. Furthermore, the points of $\pi$ in these $s$ regions form a blow-up of $\sigma$.

Furthermore, if a square in $\cS$ is not in $\Mature$, then there must be a square in $\Mature \cup \Active$ which is a subset of this square. Thus
by Lemma \ref{densesample}, with probability at least $1-\epsilon$, these $s$ regions each contains a sample point. This means, with probability at least $1-\epsilon$, $\sigma$ is a subpermutation of the sampled $\pi'$. However, since $\sigma \notin \PP$, the 1-blow-up of $\pi'$ is not in $\PP$. Since each permutation has exactly one 1-blow-up, which is itself, Algorithm 1 will output ``$\pi$ is not in $\PP$".
\end{proof}

From the previous lemma, it suffices to only further consider the case where there are not many frozen squares in any level. Formally, for each $j$, $\Frozen$ has at most $4C \cdot 2^{j-1}$ squares of side length $2^{-j}$. We have already obtained $\alpha$, a subpermutation of $\pi'$. The definition of $\alpha$ is right after Lemma \ref{densesample}, and recall $\alpha$ is induced by the points $(j_t, \pi(j_t)), t = 1,2,\dots, |\alpha|$. We now want to construct a blow-up of $\alpha$, call it $\ta$, such that $\pn(\pi,\ta)$ is small.

To define a blow-up of a given permutation, it suffices to define the length and shape of the subpermutation within each of the blow-up blocks. In the case of $\ta$, there are in total $|\alpha|$ blow-up blocks. 
The high-level idea to create an $\ta$ with small distance $\pn(\pi, \ta)$ is as follows. Recall that to bound $\pn(\pi, \ta)$, we find a matching $\psi:[n] \to [n]$ telling us that the $i$-th point of $\pi$ is matched with the $\psi(i)$-th point in $\ta$.

We first define a map $\phi: [n] \to \{j_t: 1 \leq t \leq |\alpha|\}$ which associates to each point of $\pi$ one of the sampled points in $\pi$ that make up $\alpha$. This is shown in Figure \ref{fig:phi_psi}(c), which the dashed line indicates where a point of $\pi$ is mapped to. We will use the map $\phi$ to determine $\psi$ and $\ta$. The map $\phi$ will have the property that the sum of the (taxicab) distance between points in $\pi$ and the sampled points they are mapped to is small. 
Also, if the $i$-th point of $\pi$ is associated to the $t$-th point in $\alpha$ under $\phi$, then $\psi(i)$ is in the $t$-th blow-up block in $\ta$. Thus $\phi$ determines the size of each blow-up block, which for the $t$-th blow-up block is $|\phi^{-1}(j_t)|$. To finish the construction of $\ta$, simply let the subpermutation in the $t$-th blow-up block be  the permutation induced by the subpermutation of $\pi$ restricted to the positions in $\phi^{-1}(j_t)$. This is shown in Figure \ref{fig:phi_psi}(d). To complete the construction of $\psi$, if $\psi(i), \psi(i')$ are mapped to the same blow-up block, let $\psi(i) < \psi(i')$ if and only if $i < i'$.
We will prove that the matching $\psi$ shows that the earth mover's distance $\pn(\pi, \ta)$ is small. 

We now go through the details of this argument. First we define $\phi$. 
Let $\phi: [n] \to \{j_i, i = 1, 2,\dots, |\alpha|\}$ be such that if $\phi(i) = j_t$, we associate $(i, \pi(i))$ by $\phi$ to the $t$-th point in $\alpha$. The association rule is as follows. 
\begin{enumerate}
\item If the $i$-th point of $\pi$ is in a square in $\Active$, and $(j_t, \pi(j_t)) \in \pi$ is the sample point in that square chosen to contribute to $\alpha$; we define $\phi(i) = j_t$.
\item If the $i$-th point of $\pi$ is in some sparse square in $\Frozen$, it means that its parent square was once active. If the parent square is in $\Mature$, then let $(j_t, \pi(j_t)) \in \pi$ be the unique point in that square chosen to be an element of $\alpha$ and define $\phi(i) = j_t$. If the parent square is not in $\Mature$, then after subdividing the parent square into four squares, one of them, $R$, is active. $R$ will further subdivide, and at the end of the partition procedure, either $R$ is mature, or it contains a square that is mature or active. Arbitrarily pick one such square that is mature or active, and let $(j_t, \pi(j_t))\in \pi$ be the unique element of $\alpha$ in that square and define $\phi(i) = j_t$.
\end{enumerate}
This association rule is illustrated in Figure \ref{fig:phi_psi}(c). 
We thus find for the $i$-th point in $\pi$ a target point $(\phi(t), \pi(\phi(t)))$ of $\alpha$.

For the point $(j_t, \pi(j_t)) \in \pi$ which gives the $t$-the element in $\alpha$, let $P_t = \phi^{-1}(j_t)$. Thus $P_t$ tells us which points in $\pi$ are mapped to the $t$-th point in $\alpha$. We create a blow-up $\tilde \alpha$ of $\alpha$ in the following way: the $t$-th point of $\alpha$ blows up into a block $B_t$, where $|B_t| = |\phi^{-1}(j_t)| = |P_t|$ and the permutation restricted to $B_t$ in $\ta$ is the subpermutation of $\pi$ restricted to the positions in $P_t$, i.e., $\{ (v, \pi(v)), v \in P_t \}$. By construction, $\ta$ is of length $n$. Figure \ref{fig:phi_psi}(c) and (d) illustrate how $\ta$ is obtained. 

Summarizing, we found a mapping of the points of $\pi$ to the points of a subpermutation $\alpha$. We naturally formed a blow-up $\ta$ of $\alpha$. We want to show $\pn(\pi,\ta)$ is small. Recall that by the definition of the earth mover's distance, it suffices to find a bijection between points in $\pi$ and points in $\ta$ for which the sum of the $L_1$ distances between the point in $\pi$ and the points in $\ta$ that they map to is small. 

We choose the bijection $\psi$ as follows. For any $i \neq j \leq n$, 
$\psi(i) < \psi(j)$ if and only if one of the following situation holds: (1) $\phi(i) < \phi(j)$, or (2) $\phi(i) = \phi(j)$ but $i<j$. 
Equivalently, for each $i$ (corresponding to the $i$-th point in $\pi$), let $(\psi(i), \ta(\psi(i)))$ lie in the block  $B_{\phi(i)}$. And if among all the points in $\phi^{-1}(\phi(i))$, $i$ is the $t$-th smallest, then $\psi(i)$ is the $t$-th point (from left to right) in $B_{\phi(i)}$. In other words, $\psi(i) = |B_1| + \dots + |B_{\phi(i)-1}| + t$. 

To show $\pn(\pi,\ta)$ is small, 
it suffices to show that $\sum_{i=1}^n d\left(\left(\frac{i}{n}, \frac{\pi(i)}{n}\right), \left(\frac{\psi(i)}{n}, \frac{\ta(\psi(i))}{n}\right)\right)$ is small, where  $d((x_1, y_1), (x_2, y_2)) = |x_1 - x_2| + |y_1 - y_2|$ is the $L_1$ distance. The next lemma shows that this is within a factor two of the sums of the $L_1$ distance between the points in $\pi$ and the points in $\alpha$ that they map to. This in turn is easier to bound as the $L_1$ distance between two points in the same axis-parallel square is at most twice the side length of the square.

\begin{lem}\label{sf1}
\[ \sum_{i=1}^n d\left(\left(\frac{i}{n}, \frac{\pi(i)}{n}\right), \left(\frac{\psi(i)}{n}, \frac{\ta(\psi(i))}{n}\right)\right) \leq 2\sum_{i=1}^n d\left(\left(\frac{i}{n}, \frac{\pi(i)}{n}\right), \left(\frac{\phi(i)}{n}, \frac{\pi(\phi(i))}{n}\right)\right).\]
\end{lem}

We first prove the following useful claim. 
\begin{claim}[Generalized Spearman's Footrule]
Let $\gamma(i): [n] \to [n]$ be any map. 
Define a bijection $\beta:[n] \rightarrow [n]$ as follows. We have $\beta(i) < \beta(j)$ if and only if (1) $\gamma(i) < \gamma(j)$ or (2) $\gamma(i) = \gamma(j)$ and $i<j$.

Let $\D(\beta) = \sum_{i=1}^n |i-\beta(i)|$, which is Spearman's footrule distance between the identity permutation and the permutation $\beta$. Let $\D(\gamma) = \sum_{i=1}^n |i - \gamma(i)|$, which is a footrule distance between the identity permutation and the mapping $\gamma$.
Then
\[ \D(\beta)  \leq 2 \D(\gamma).\]
\end{claim}
\begin{proof}[Proof of the claim]
For each integer $1\leq j \leq n$, let $\gamma_j^+$ be the number of integers $i$ such that $ i \leq j < \gamma(i)$. Similarly, let $\gamma_j^-$ be the number of $i$ such that $\gamma(i) \leq j < i$. Define $\beta_j^+, \beta_j^-$ similarly. 

We first show that $\D(\gamma) = \sum_{j=1}^n \gamma_j^+ + \gamma_j^-$. We use a double counting argument.
Note that if $i \leq \gamma(i)$, then the number of $j$ for which $i$ contributes one to $\gamma_j^+$ is precisely $|i-\gamma(i)|$, and the number of $j$ for which $i$ contributes to $\gamma_j^-$ is zero. Similarly, if $i > \gamma(i)$, then the number of $j$ for which $i$ contributes one to  $\gamma_j^-$ is $|i-\gamma(i)|$, and the number of $j$ for which $i$ contributes to $\gamma_j^+$ is zero. Therefore, $\D(\gamma) = \sum_{i: i < \gamma(i)} |i-\gamma(i)| + \sum_{i: i > \gamma(i)} |i-\gamma(i)| = \sum_{j=1}^{n}  \gamma_j^+ + \sum_{j=1}^{n}  \gamma_j^-$. 
Similarly, we have $\D(\beta) = \sum_{j=1}^n \beta_j^+ + \beta_j^-$.

To prove $\D(\beta)  \leq 2 \D(\gamma)$, 
it suffices to show for any $j$, $ \beta_j^+ \leq \gamma_j^+ + \gamma_j^-$ and $\beta_j^- \leq \gamma_j^+ + \gamma_j^-$.
The proofs of both inequalities are essentially the same argument, so we only show the first. If
$i$ contributes one to $\beta_j^+$, then $i \leq j < \beta(i)$. If $\gamma(i)>j$, then $i$ also contributes one to $\gamma_j^+$. Otherwise, $i,\gamma(i) \leq j < \beta(i)$. Since $\beta$ is a permutation, for the set of such $i$, since $i \leq j < \beta(i)$, we can match each such $i$ with a unique $i'$ such that $\beta(i') \leq j < i'$. However, since $\beta(i') < \beta(i)$ and $i' > i$, we must have $\gamma(i') < \gamma(i)$, thus $\gamma(i') < j <i'$, which implies that $i'$ contributes one to $\gamma_j^-$. We thus get $\beta_j^+ \leq \gamma_j^+ + \gamma_j^-$. Similarly, $\beta_j^- \leq \gamma_j^+ + \gamma_j^-$. Summing over all $j$, we get $\D(\beta)  \leq 2 \D(\gamma)$.
\end{proof}
\begin{proof}[Proof of Lemma \ref{sf1}]
It suffices to prove the following two inequalities.
\[ \sum_{i=1}^n \left|\frac{i}{n} -  \frac{\psi(i)}{n}\right| \leq 2\sum_{i=1}^n \left|\frac{i}{n}-\frac{\phi(i)}{n}\right|.\]
\[ \sum_{i=1}^n  \left|\frac{\pi(i)}{n}-\frac{\ta(\psi(i))}{n}\right| \leq 2\sum_{i=1}^n \left|\frac{\pi(i)}{n}-\frac{\pi(\phi(i))}{n}\right|.\]
To prove the first inequality, 
we, apply the claim with $\gamma(i) = \phi(i)$ and $\beta(i) = \psi(i)$.  By the definition of $\psi$ which is 
$\psi(i) < \psi(j)$ if and only if (1) $\phi(i) < \phi(j)$, or (2) $\phi(i) = \phi(j)$ but $i<j$, the condition of the claim holds. 
For the second inequality, we apply the claim with $\gamma(i) = \pi\circ \phi\circ \pi^{-1}(i)$ and $\beta(i) = \ta \circ \psi\circ \pi^{-1}(i)$. Similarly by the definition of $\ta$, the conditions for the claim holds. 
\end{proof}

\begin{lem}\label{moveshort}
Let $\PP$ be a proper hereditary family of permutations, $\sigma$ be a smallest permutation not in $\PP$, and $s = |\sigma|$. Let $C = 2c(\sigma)$, where $c(\sigma)$ is the F\"uredi-Hajnal constant for $\sigma$. Assume, for permutation $\pi$, that for each $j$, there are at most $4C \cdot 2^{j-1}$ frozen squares of side length $2^{-j}$. Further suppose that the random subpermutation $\alpha$ of $\pi$ contains at least one point in each rich square. Then 
\[ \sum_{i=1}^n d\left(\left(\frac{i}{n}, \frac{\pi(i)}{n}\right), \left(\frac{\phi(i)}{n}, \frac{\pi(\phi(i))}{n}\right)\right) \leq 2dn + 8\log(1/d)C\delta n.\]
\end{lem}
\begin{proof}
For any point in an active square, the $L_1$ distance between it and the point of $\alpha$ it maps to is at most $2d$ since they are both in the same active square, which has side length $d$ by Lemma \ref{procedurec}. There are at most $n$ points of $\pi$ which are in active squares. Thus, the sum of the distances coming from points in active squares is at most $2dn$.

For a point $(i/n, \pi(i)/n)$ in a frozen square $R \in \Frozen$, suppose the side length of the square is $1/2^j$. The $L_1$ distance from it to the point in $\alpha$ it associates to is at most $2\cdot 1/2^{j-1}$ since these two points are in the same parent square whose side length is $1/2^{j-1}$. However, by assumption there are at most $4C \cdot 2^{j-1}$ frozen squares of side length $2^{-j}$. Therefore for each $j$, the sum of $L_1$ distances coming from points in frozen squares of side length  $1/2^j$ is at most
$(4C \cdot 2^{j-1}) \cdot  \delta n \cdot 2/2^{j-1}$, 
where $(4C \cdot 2^{j-1}) \cdot  \delta n$ gives an upper bound for the number of points in frozen squares of side length $1/2^{j}$ and $2/2^{j-1}$ is an upper bound for the $L_1$ distances between a point in $\pi$ and its image in $\alpha$. 
Therefore the sum of the distances coming from points in frozen squares of different side lengths is at most
\[ \sum_{j=1}^{\log(1/d)} 8C \delta n  = 8\log(1/d)C\delta n.\]

Combining the distances from points in active squares and points in frozen squares, we have that
\[ \sum_{i=1}^n d\left(\left(\frac{i}{n}, \frac{\pi(i)}{n}\right), (a_{\phi(i)}, b_{\phi(i)})\right) \leq 2dn + 8\log(1/d)C\delta n .\]
\end{proof}

Lemmas \ref{sf1} and \ref{moveshort}, together with the values of $\delta, d$ as given in (\ref{sec3delta}) and (\ref{sec3d}) give the following lemma.
\begin{lem}\label{pita}
Assume that for each $j$, there are at most $4C \cdot 2^{j-1}$ frozen squares of side length $2^{-j}$, and the sample contains at least one point in each rich square. Then 
\[ \pn(\pi, \ta) \leq 2(2d + 8\log(1/d)C\delta) < \epsilon/2.\]
\end{lem}

We have shown $\ta$ and $\pi$ are close in the earth mover's distance. We  now show that $\ta$ is close in the earth mover's distance to any other blow-up of $\alpha$, where the $j$-th point in $\alpha$ blows up into a block of size $|\phi^{-1}(j)| = |B_j|$. Thus $\pi$ and any of these blow-ups are close.

\begin{lem}\label{tata'}
Let $\ta'$ be any blow-up of $\alpha$ where the $j$-th point in $\alpha$ blows up into a block of size $|\phi^{-1}(j)| = |B_j|$. Then 
\[ \pn(\ta', \ta) \leq \epsilon/2.\]
\end{lem}

 Before proving Lemma \ref{tata'}, we first obtain the following useful claim.
\begin{claim}
\[ \delta n \leq |B_j| \leq 3\delta n \log(1/d) + \max(dn, 4\delta n).\]
\end{claim}
\begin{proof}
By definition, $B_j = \phi^{-1}(j)$. 
Since $(t_j, \pi(t_j))$ is in an active or mature square, and all the other points in this square are mapped to it by $\phi$; we have $|B_j| \geq \delta n$. 

Given any mature or active square $R$ of side length $l$, by the partition procedure, we know for each possible length $l' > l$, there are at most three frozen squares of side length $l'$ whose points are mapped to a sample in $R$. Since there are at most $\log(1/l)$ possible side lengths of squares greater than $l$, the number of points coming from the frozen squares mapped to the sample point in $R$ is at most $\delta n \cdot (3 \log(1/l)) \leq 3\delta n \log(1/d)$.

Furthermore, by Lemma \ref{procedurec}, the number of points in $R$ is at most $\max(dn, 4\delta n)$. The points in $R$ and points in a frozen square mapped to $R$ are all the possible points mapping to $R$. Thus $|\phi^{-1}(j)| \leq 3\delta n \log(1/d) + \max(dn, 4\delta n).$
\end{proof}

\begin{proof}[Proof of Lemma \ref{tata'}]
By the assumption, we know $\ta, \ta'$ have the same corresponding blow-up block sizes. To get an upper bound for $\pn(\ta' ,\ta)$, we simply pick the identity permutation $\theta$ given by $\theta(i)=i$ for $i \in [n]$. We thus would like to upper bound
\[ \sum_{i=1}^n d((i/n, \ta(i)/n), (i/n, \ta'(i)/n) = \sum_{i=1}^n |(\ta(i)-\ta'(i))/n|.\]
Notice that if the $i$-th point in $\ta$ is in the $j$-th blow-up block, then the $i$-th point in $\ta'$ is also in the $j$-th block. Thus $|\ta(i) - \ta'(i)| \leq |B_j|$. 
This gives us
\beq    \sum_{i=1}^n |(\ta(i)-\ta'(i))/n| = \sum_{j=1}^{|\alpha|} \left( \sum_{i: (i, \ta(i)) \in B_j}  |(\ta(i)-\ta'(i))/n| \right) \leq \sum_{j=1}^{|\alpha|} |B_j|^2 /n. \label{eq:alphaalpha'} \eeq

We have $\sum_{j=1}^{|\alpha|} |B_j| = n$, and $0 \leq |B_j| \leq 3\delta n \log(1/d) + \max(dn, 4\delta n)$. To maximize $\sum_{j=1}^{|\alpha|}  |B_j|^2 /n$ with these constraints, we want the $|B_j|$'s as  close to the extreme values as possible. This gives $\sum_{j=1}^{|\alpha|} |B_j|^2 /n \leq (3\delta  \log(1/d) + \max(d, 4\delta )) n$ where we assign $n/(3\delta n \log(1/d) + \max(dn, 4\delta n))$ number of $|B_j|$'s to achieve the maximum possible value $3\delta n \log(1/d) + \max(dn, 4\delta n)$, and the rest are set to be 0. 

Therefore, we have 
\[ \pn(\ta',\ta) \leq \frac{1}{\frac{n-1}{2}} \sum_{i=1}^n |(\ta(i)-\ta'(i))/n| \leq \frac{n}{(n-1)/2}(3\delta  \log(1/d) + \max(d, 4\delta)) \leq \epsilon/2.\]
The first inequality is by the definition of $\pn$, and the last inequality is by our choices of $\delta, d$  as in (\ref{sec3delta}) and (\ref{sec3d}). 
\end{proof}

Combining the previous two lemmas, we have the following lemma. 

\begin{lem}\label{closeto}
Suppose the subpermutation $\alpha$ contains at least one point in each rich square. Let $\ta'$ be any blow-up of $\alpha$ such that the $j$-th point in $\alpha$ blows up into a block of size $|\phi^{-1}(j)| = |B_j|$. 
Assume that for each $j$, there are at most $4C \cdot 2^{j-1}$ frozen squares of side length $2^{-j}$. Then 
\[ \pn(\pi, \ta') < \epsilon.\]
\end{lem}
\begin{proof}
This is a direct consequence of Lemmas \ref{tata'} and \ref{pita} and the triangle inequality: \[\pn(\pi, \ta') \leq \pn( \pi , \ta) + \pn(\ta , \ta') < \epsilon.\]
\end{proof}

Assuming Lemma \ref{closeto}, we show that the second assertion of Theorem \ref{newm} holds.
\begin{proof}
We have obtained that with probability at least $1-\epsilon$, we can obtain a permutation $\alpha$ being a subpermutation of the sample permutation $\pi'$, and $\alpha$ hits each of the dense squares in $\Rich$ once. 
If our algorithm outputs ``$\pi \in \PP$", it implies that there is an $n$-blow-up of $\alpha$ in $\PP$. Therefore we can find a subpermutation $\ta'$ of the $n$-blow-up, where $\ta'$ is of length $n$, and it is a blow-up of $\alpha$ where the $j$-th point of $\alpha$ blows up into a block of size $|B_j|$. (Notice that this $\ta'$ may be different from the $\ta$ defined right before Lemma \ref{sf1}.) Since $\PP$ is hereditary, $\ta' \in \PP$.
However, by Lemma \ref{closeto}, we know  $\pn(\pi, \ta')  < \epsilon$.
This contradicts the assumption that $\pi$ is at distance at least $\epsilon$ from $\PP$ and completes the proof. 
\end{proof}

\section{Two-sided property testing under the planar metrics with universal bound}
In this section we prove Theorem \ref{maincut}. The tester we use is the same as the two-sided tester in Algorithm \ref{tester:1}, but with a different choice of $M$. 
The proof is similar to that of Theorem \ref{maincutnearlylinear}, but with a somewhat different analysis. When $\pi \in \PP$, by Lemma \ref{typicalsample} and the same argument as in the proof of Theorem \ref{maincutnearlylinear}, it can be seen that when $n$ is large enough (depending on $\PP$), the algorithm outputs $\pi \in \PP$ with probability at least $1-\epsilon$. 

We now handle the case when $\pi$ is $\epsilon$-far from $\PP$. 
Let 
\beq M= 20000/\epsilon^2 \label{sec4M} \eeq
 be the sample size. 
We will first prove the following lemma, which states that any permutation is very close to some blow up of a large enough sample with high probability. 

\begin{lem}\label{closetobu}
Let $\pi$ be a permutation of length $n$ and $M$ be as in (\ref{sec4M}). An $M$-sample chosen uniformly at random in $\pi$, given by the sample points $(j_t, \pi(j_t))$ for $t = 1,2,\dots, M$, induces a permutation $\alpha$. Then, with probability at least $1-\epsilon/2$, there exists a blow up of $\alpha$ which is within distance $\epsilon/2$ (under $\pn$) with $\pi$.
\end{lem}
\begin{rem}
The sample size $M = O(1/\epsilon^2)$ is tight. The following is the main idea of the proof. For a permutation of length $n$ picked uniformly at random with $n \gg \epsilon^{-2}\log(1/\epsilon)$, if $M <2^{-7}\epsilon^{-2}$, then a square of side length $8\epsilon$ around a point of $\pi$ with probability at least $2/3$ does not contain any sample point.  Thus, we expect that most points have to move distance more than $2\epsilon$ in any blow-up of $\alpha$.  
\end{rem}
To prove Lemma \ref{closetobu}, we proceed similar to the proof of Theorem \ref{maincutnearlylinear}.  Define a map $\phi: [n] \to [n]$ which maps each point of $\pi$ to the closest (under the  Euclidean distance\footnote{Note that one can also easily use the $L_\infty$ or $L_1$ distance as they are Lipschitz equivalent; the proofs are basically the same but it is easier to explain it with figures using the Euclidean distance.} in $[0,1]^2$) sample point. 
 In other words, $\phi(i) = j$ is equivalent to that the point $(i, \pi(i))$ of $\pi$ is closest to the sample point $(j, \pi(j))$.  Similar to the proof of Theorem \ref{maincutnearlylinear}, from $\phi$ we also define a permutation $\ta: [n] \to [n]$ which is a blow-up of $\alpha$.  For each $t = 1,2,\dots, M$, let $P_t = \phi^{-1}(j_t)$; then the permutation restricted to the $t$-th blow-up block $B_t$ of $\ta$ is the subpermutation of $\pi$ restricted to the positions indicated in $P_t$. Clearly $|B_t| = |P_t|$. We also define a matching $\psi: [n] \to [n]$, such that for any $i \neq j \leq n, \psi(i) < \psi(j)$ if and only if one of the following holds: (1) $\phi(i) < \phi(j)$, or (2) $\phi(i) = \phi(j)$ and $i<j$. Thus $\psi$ gives a matching between points in $\pi$ and points in $\psi$. Similar as before, we will show that with probability at least $1 - \epsilon/2$, 
\beq \pn(\pi, \ta) < \epsilon/2, \label{eq:1}\eeq
which is a stronger statement  of Lemma \ref{closetobu};
We will further prove that with probability at least $1 - \epsilon/2$, if $\ta'$ is another blow-up of $\alpha$ where the $j$-th point of $\alpha$ blows up into a block of size $|B_j|$ as well, then 
 \beq \pn(\ta, \ta')< \epsilon/2.\label{eq:2}\eeq
  Inequalities (\ref{eq:1}) and (\ref{eq:2}) and the triangle inequality imply that $\pn(\pi, \ta') < \epsilon$. Assuming we can prove that $\pn(\pi, \ta') < \epsilon$, if the algorithm outputs ``$\pi \in \PP$'', then there is an $n$-blow-up of $\alpha$ in $\PP$. Therefore we can find a length-$n$ subpermutation $\ta'$ of this $n$-blow-up where the $j$-th point of $\alpha$ blows up into a block of size $|B_t|$. Clearly $\ta' \in \PP$ by the hereditary property. Since we have assumed $\pn(\pi, \ta') < \epsilon$, we reach a contradiction since we also assume $\pn(\pi, \PP) \geq \epsilon$.

Therefore we just need to prove two things: Lemma \ref{closetobu} or equivalently inequality (\ref{eq:1}), and (\ref{eq:2}). Note that as a consequence inequalities (\ref{eq:1}) and (\ref{eq:2}) will simultaneously hold with probability at least $1-\epsilon$. 
The proof consists of two parts. The first part is to show inequality (\ref{eq:1}), and the second part is to show inequality (\ref{eq:2}). The first part is done through the dyadic partition procedure again. 

\subsection{$\pn(\pi, \ta) < \epsilon/2$}
Similar to in the proof of Theorem \ref{maincutnearlylinear}, we again conduct a dynamic procedure in different levels, starting from level 0, where in level $i$ we look at squares of side lengths $2^{-i}$. Recall that the density of a square is the fraction of points in the permutation in the square (so it is not the relative density with respect to the area of the square). Let 
\beq d = 2^{-\lceil \log(48/\epsilon) \rceil} \label{sec4d} \eeq 
be the smallest side length of squares. Let $l_0 = \log(1/d)$ be the number of different levels. Let 
\beq  \delta_1 = \sqrt{2}/96 \cdot \epsilon^2 \log(1/\epsilon),  \ \ \delta_2 = \epsilon^2 / 5200, \ \ K = \lceil- \log(12 \cdot 16 \cdot \delta_1/\epsilon) \rceil  \label{sec4delta} \eeq
where $\delta_1, \delta_2$
are two density thresholds; If $i \leq K$, then a square of side length $2^{-i}$ is in $\Active$ if it has density at least $\delta_1$; if $i > K$, then a square of side length $2^{-i}$ is in $\Active$ if it has density at least $\delta_2$. We go through the same procedure as in Algorithm \ref{partitionnm}, with the only difference being that the density threshold for a square to be in $\Active$ depends on its side length. In summary, throughout the procedure, repeatedly some parent square in $\Active$ splits into four children squares, and a child square splits if is dense with respect to the corresponding density threshold.

For each level $0 \leq i \leq \log(1/d)$, let $\Frozen_i$ to be the set of squares in $\Frozen$ which are of side length $2^{-i}$. Let $\Mature_i$ be the set of squares in $\Mature$ which are of side length $2^{-i}$. Note that each square in $\Mature_i$ is a union of four squares in $\Frozen_{i+1}$. Let $\Active_i$ be the squares of side length $2^{-i}$ that are once active. Clearly $\Active_{l_0} = \Active$.

We will define a map $\phi': [n] \to [n]$ which maps points of $\pi$ to some sample point. First notice that $\phi$ is the map which maps each point of $\pi$ to its closest sample point; thus 
\beq \sum_{i=1}^n \left|\frac{i}{n} -  \frac{\phi(i)}{n}\right|+  \left|\frac{\pi(i)}{n} -  \frac{\pi(\phi(i))}{n}\right| \leq  \sum_{i=1}^n \left|\frac{i}{n} -  \frac{\phi'(i)}{n}\right|+  \left|\frac{\pi(i)}{n} -  \frac{\pi(\phi'(i))}{n}\right|.  \label{phi'phi} \eeq
We now define $\phi'$ such that we can bound the right hand side of (\ref{phi'phi}), and thus upper bounds the left hand side of (\ref{phi'phi}) which is about $\phi$. 

For a point $(i, \pi(i))$ in $\pi$, let $1 \leq j \leq l_0$ be the largest integer such that there exists a square in $\Active_j$ containing both the point $(i, \pi(i))$ of $\pi$ and some sample points. Let $\phi'(i) \in \{ j_t, t = 1,2,\dots, M\}$ be the position of a closest such sample point to $(i, \pi(i))$. In other words, we find the smallest square which was once in $\Active$ containing this point and some sample points and $\phi'$ maps $(i, \pi(i))$ to a closest such sample point.

To analyze the right hand side in (\ref{phi'phi}), 
we first partition points in $\pi$ into different groups as follows. 
For each $1 \leq i \leq \log(1/d)$, we define a set of squares $\Can_i$ as follows:\\
(1) if $0 \leq i < \log(1/d)$, $\Can_i$ consists of squares in $\Mature_i \cup \Frozen_{i+1}$. \\ (Notice $[0,1]^2$ cannot be in $\Mature$). \\ (2) if $i = \log(1/d)$, $\Can_i$ consists of squares in $\Active$, which by definition are of side length $d$.  \\ Denote the number of points of $\pi$ in $\Can_i$ as $n_i$. Thus $\sum_{i=0}^{\log(1/d)} n_i = n$. 
The following lemma is easy to check from the definitions.
\begin{lem}\label{partprop2}
\begin{enumerate}[(a)]
\item The squares in the sets $\Can_i$ partition $[0,1]^2$.
\item Given $1 \leq i \leq \log(1/d)$, the points in squares in $\bigcup_{i \leq j \leq \log(1/d)} \Can_j$ are the same as the points squares in $\Active_i$; The points of $\pi$ in squares in $\Can_i$ are the points in squares in $\Active_i$ but not in squares in $\bigcup_{j > i} \Active_j$.
\item For each $K+1 \leq i < \log(1/d)$, $n_i \leq 4\delta_2 \cdot 2^{2i}n$; for each $0 \leq i \leq K$, $n_i \leq 4\delta_1 \cdot 2^{2i}n$;
\end{enumerate}
\end{lem}
\begin{proof}
The first two assertions are clear from the definition of $\Can_i$. 

To show (c) holds, notice that when $i < \log(1/d)$, $\Can_i$ consists of only squares in $\Frozen_{i+1}$ or $\Mature_i$. If $i \leq K$, each dyadic square of side length $2^{-i}$ has at most $4\delta_1n$ points of $\pi$ coming from squares in $\Mature_i$ or $\Frozen_{i+1}$, and there are $2^{2i}$ dyadic squares of side length $2^{-i}$. Thus the bound is achieved. Similar argument applies to the case when $i >K$. 
\end{proof}

 For a point in a square in $\Can_i$, we say that it is \emph{resolved in level $j$}, if $j$ is the largest integer such that there is a square in $\Active_j$ containing this point and a sample point. Clearly by the definition of $\Can_i$, $j \leq i$. And by the definition of $\phi'$ defined above, if a point $(s, \pi(s))$ of $\pi$ is resolved in level $j$, then
\beq
| s - \phi(s)| + |\pi(s) - \pi(\phi(s))| \leq | s - \phi'(s)| + |\pi(s) - \pi(\phi'(s))| \leq 2\cdot 2^{-j}. \label{phi'dis}
\eeq

Let the random variables $N_{i,j}$ to be the set of points in squares in $\Can_i$ which are resolved in level $j$. Let $n_{i,j} = |N_{i,j}|$. Thus $\sum_{j: 0\leq j \leq i}   n_{i,j} = n_i$, the number of points of $\pi$ in squares in $\Can_i$.  Let $\NR_j$ be the set of points in squares in $\bigcup_{l_0 \geq i \geq j}\Can_i$ which are not resolved in levels $l_0, l_0-1, \dots, j$. Thus by the definition of $\NR_j$, 
\beq |\NR_j |= \sum_{ j \leq s \leq l_0} \sum_{0 \leq t \leq j-1} n_{s,t}. \label{nscard}
\eeq

We first show that under some conditions, the quantity on the RHS of (\ref{phi'phi}) is indeed small; and then we show that with high probability, these conditions hold. 
\begin{lem}\label{moveshort2}
Let $c=1/2$. Suppose the following conditions holds simultaneously.
\begin{enumerate}[(i)]
\item for each $i \leq K$, points in squares in $\Can_i$ are resolved in level $i$; 
\item for each $i> K$, a point in a square in $\Can_i$ is resolved in level $j$ with $j \geq K$;
\item for each $i> K$, $|\NR_j| \leq c \delta_2 2^{2j}n$. 
\end{enumerate}
Then under these three conditions, 
\[ \sum_{i=1}^n \left|\frac{i}{n} -  \frac{\phi(i)}{n}\right|+  \left|\frac{\pi(i)}{n} -  \frac{\pi(\phi(i))}{n}\right|  \leq  \left( 16\delta_1 2^{K} +  \frac{8\delta_2}{d} + 2 \cdot d  + \frac{2c\delta_2}{d} \right) n.\]
\end{lem}
\begin{proof}
By (\ref{phi'dis}) and the definition of $n_{i,j}$, we have the bound
\begin{align}
& \sum_{i=1}^n \left|\frac{i}{n} -  \frac{\phi'(i)}{n}\right|+  \left|\frac{\pi(i)}{n} -  \frac{\pi(\phi'(i))}{n}\right| \nonumber \\
 \leq & \sum_{\log(1/d) \geq  i \geq j \geq 0} 2 \cdot 2^{-j} n_{i,j} \nonumber \\
 = &  \sum_{i=0}^K\sum_{j=0}^i 2\cdot 2^{-j} n_{i,j} +  \sum_{i=K+1}^{l_0} \sum_{j=0}^i 2\cdot 2^{-j} n_{i,j}  . \label{phi'dis2}
\end{align}
The first inequality is by (\ref{phi'dis}) and the definition of $n_{i,j}$; and the equality is simply regrouping the terms into two groups. 
By conditions (i) and (ii), we have for $0 \leq i \leq K$, $n_{i,i} = n_i$ and $n_{i,t} = 0$ for $t < j$.  For $i \geq K+1$, $n_{i,t} = 0$ for $t<K$. Thus combining with Lemma \ref{partprop2}(c) for the bound on $n_i$, the first summand in (\ref{phi'dis2}) can be bounded by
\beq   \sum_{i=0}^K \left(\sum_{j=0}^i  2\cdot 2^{-j} n_{i,j} \right) =   \sum_{i=0}^K 2\cdot 2^{-i} n_i \leq   \sum_{i=0}^K2\cdot 2^{-i} 4\delta_1 2^{2i}n ={8\delta_1}  \sum_{i=0}^K 2^{i}n \leq {16\delta_1}2^{K}n .  \label{firstsum}
\eeq
We now bound the second summand in (\ref{phi'dis2}). Note that $|\NR_K| = 0$ by conditions (i) and (ii). Furthermore, by (\ref{nscard}), 
\beq
|\NR_j| - |\NR_{j-1}| = n_{j,{j-1}} + \dots + n_{{l_0},{j-1}} - (n_{j-1,j-2} + \dots + n_{j-1,K}). \label{nrcarddiff}
\eeq
We now rewrite the second summand in (\ref{phi'dis2}).
\begin{align}
&  \sum_{i=K+1}^{l_0} \sum_{0 \leq j \leq i} 2\cdot 2^{-j} n_{i,j} \nonumber  \\
 = &
 \sum_{j=K+1}^{l_0} 2\cdot 2^{-j}n_{j,j} +  \sum_{t=K}^{l_0-1}\sum_{j=t+1}^{l_0} 2\cdot 2^{-t}n_{j,t}   \nonumber\\
= &  \sum_{j=K+1}^{l_0} 2\cdot 2^{-j}n_{j,j}  + \sum_{t=K}^{l_0-1} 2\cdot 2^{-t} (|\NR_{t+1}| - |\NR_t| + n_{t,t-1} + \dots + n_{t,K})  \nonumber \\
= & \sum_{j=K+1}^{l_0} 2\cdot 2^{-j}( n_{j,j}  + n_{j,j-1} + \dots + n_{j,K})+ \sum_{ t=K}^{l_0-1} 2\cdot 2^{-t} (|\NR_{t+1}| - |\NR_t|) \nonumber \\
= & \sum_{j=K+1}^{l_0} 2\cdot 2^{-j} n_j+ \sum_{t=K}^{l_0-1} 2\cdot (2^{-(t-1)} - 2^{-t}) |\NR_t| \nonumber \\
= &  \sum_{j=K+1}^{l_0} 2\cdot 2^{-j} n_j+  \sum_{t=K}^{l_0-1} 2\cdot 2^{-t} |\NR_t|  \nonumber\\
\leq &  \sum_{j=K+1}^{l_0-1} 2\cdot 2^{-j} (4\delta_2 2^{2j}n)+2\cdot d n+  \sum_{t=K}^{l_0-1} 2\cdot 2^{-t} (c\delta_2 2^{2t}n)  \label{secondsum}
\end{align}
where the second equality is by (\ref{nrcarddiff}), the first inequality is by Lemma \ref{partprop2}(c) and $n_{l_0} \leq n$, and condition (iii) in the lemma statement.
Using a simple upper bound on a geometric series, (\ref{secondsum}) is bounded above by
\beq
{8\delta_2} \cdot 2^{l_0}n+ 2 d n + {2c\delta_2}\cdot 2^{l_0}n = \frac{8\delta_2}{d} n+ 2 dn  + \frac{2c\delta_2n}{d}. \label{secondsum2}
\eeq
Therefore, combining (\ref{firstsum}) and (\ref{secondsum2}), we have
\beq
 \sum_{i=1}^n \left|\frac{i}{n} -  \frac{\phi'(i)}{n}\right|+  \left|\frac{\pi(i)}{n} -  \frac{\pi(\phi'(i))}{n}\right|
 \leq  \left( {16\delta_1}2^{K} +  \frac{8\delta_2}{d} + 2 \cdot d  + \frac{2c\delta_2}{d}  \right) n. \label{phi'dis2'}
\eeq
Thus, by (\ref{phi'phi}), the claim holds.
\end{proof}

Now we show that with probability at least $1-\epsilon/2$, the conditions in Lemma \ref{moveshort2} hold. 
\begin{lem}\label{condi}
Let $c=1/2$. With probability at least $1-\epsilon/2$, the following conditions holds simultaneously.
\begin{enumerate}[(a)] %
\item for each $i \leq K$, points in squares in $\Can_i$ are resolved in level $i$; 
\item for each $i> K$, a point in a square in $\Can_i$ is resolved in level $j$ with $j \geq K$;
\item for each $i> K$, $|\NR_j| \leq c \delta_2 2^{2j}n$. 
\end{enumerate}
\end{lem}
\begin{proof}
For each $0 \leq i \leq K$, and for each square in $\Active_i$, we know it has at least $\delta_1n$ points by the definition of $\Active_i$. The probability that this square does not contain a sample point is bounded above by $(1-\delta_1)^M \leq e^{-\delta_1 M}$. There are at most $2^{2i}$ squares in $\Active_i$, therefore by the union bound, the probability that the sample misses some square in $\Active_i$ is at most $2^{2i} e^{-\delta_1M}$. By the union bound again, the probability that there is an $0\leq i \leq K$ such that the sample misses some square in $\Active_i$ is bounded above by 
\beq
\sum_{i=0}^K 2^{2i} e^{-\delta_1M} < 2^{2K+1} e^{-\delta_1 M} \leq \epsilon/4. \label{probsmall}
\eeq
where the last inequality is by our choices of $\delta_1, K$ as in (\ref{sec4delta}).

Now we consider the case $K+1 \leq i \leq l_0$. Recall $\NR_i$ are the points in squares in $\Active_i$ which do not contain a sample point. Each square of side length $2^{-i}$ in $\Active_i$ has at least $\delta_2 n$ points of $\pi$ by definition. Suppose $|\Active_i| = T_i$. We know $|T_i| \leq 2^{2i}$ since there are at most $2^{2i}$ squares of side length $2^{-i}$. Given any set of squares in $\Active_i$ which consist of at least $c\delta_2 2^{2i}n$ points of $\pi$, the probability that these squares do not contain a sample point is at most $(1- c\delta_2 2^{2i})^M \leq e^{-c\delta_2 2^{2i} M}$. Since there are at most $2^{|T_i|} \leq 2^{2^{2i}}$ choices for these squares, by the union bound, 
the probability that some squares in $\Active_i$ containing at least $c\delta_2 2^{2i}n$ points and do not contain a sample point is bounded above by $2^{ 2^{2i}} e^{-c\delta_2 2^{2i} M}$. Therefore
 \begin{align}
 & \Pr(|\NR_i| \geq c\delta_2 2^{2i}n) \leq 
 2^{ 2^{2i}} e^{-c\delta_2 2^{2i} M} = e^{c 2^{2i}(c' - \delta_2 M)} \label{probbig}
\end{align}
where $c' =
 \ln(4e)$. 
By the union bound, the probability that $|\NR_i| \geq c\delta_2 2^{2i}n$ for some $K+1 \leq i \leq l_0$ is bounded above by $\sum_{i = K+1}^{l_0}   e^{c 2^{2i}(c' - \delta_2 M)}.$
Since $c' -\delta_2 M  = \ln(4e) - 20000/5200 = -1.46$ by the values of $M$ and $\delta_2$ as in (\ref{sec4M}) and (\ref{sec4delta}),  we have
\begin{align}
 & \Pr(|\NR_i| \geq c\delta_2 2^{2i}n \text{ for some } K+1 \leq i \leq l_0) \\
\leq 
&  \sum_{i = K+1}^{l_0}   e^{c 2^{2i}(c' - \delta_2 M)} 
<   \sum_{j = 2^{2(K+1)}}^{\infty}   e^{c j(c' - \delta_2 M)}   \nonumber \\
 = &  e^{c(c' - \delta_2 M) (2^{2(K+1)}-1)} / (e^{-c(c' - \delta_2 M)} -1) < \epsilon/8. \label{probcase2}
\end{align}
Thus the claim in the lemma holds by (\ref{probsmall}) and (\ref{probcase2}). 
\end{proof}

Combining Lemma \ref{moveshort2} and Lemma \ref{condi}, we know that with probability at least $1-\epsilon/2$, 
\[ \sum_{i=1}^n \left|\frac{i}{n} -  \frac{\phi(i)}{n}\right|+  \left|\frac{\pi(i)}{n} -  \frac{\pi(\phi(i))}{n}\right|  \leq \left( 16\delta_1 2^{K} +  \frac{8\delta_2}{d} + 2 \cdot d  + \frac{2c\delta_2}{d} \right) < \epsilon n/4.\]
Recall the definitions from earlier: $\ta$ is the blow-up of the subpermutation $\alpha$ induced by the sample in $\pi$ where the $t$-th blow-up block is the subpermutation induced by the locations $\phi^{-1}(j_t)$ in $\pi$; and $\psi$ is the bijection between locations of $\pi$ and $\ta$ with $\psi(i) < \psi(j)$ if and only if $\phi(i) < \phi(j)$ or $\phi(i)  = \phi(j)$ but $i<j$. 
 Then for the bijection $\psi$ and the blow-up permutation $\ta$, similar to the proof before, we have by Lemma \ref{sf1} that
\begin{align*}
\pn(\pi, \ta)  &\leq  \frac{1}{(n-1)/2} \sum_{i=1}^n d\left(\left(\frac{i}{n}, \frac{\pi(i)}{n}\right), \left(\frac{\psi(i)}{n}, \frac{\ta(\psi(i))}{n}\right)\right)  \\
& \leq \frac{2}{(n-1)/2}\sum_{i=1}^n d\left(\left(\frac{i}{n}, \frac{\pi(i)}{n}\right), \left(\frac{\phi(i)}{n}, \frac{\pi(\phi(i))}{n}\right)\right) < \epsilon/2.
\end{align*}

\subsection{$\pn(\ta, \ta') < \epsilon/2$}
In this section we show the following lemma.
\begin{lem}\label{tata'2}
With probability at least $1-\epsilon/2$, any blow-up $\ta'$ of $\alpha$ with the $j$-th point of $\alpha$ blows up into a block of size $|\phi^{-1}(j)| = |B_j|$, satisfies
\[
\pn(\ta', \ta) < \epsilon/2.
\]
\end{lem}
\begin{proof}
As in the proof of Lemma \ref{tata'}, we again obtain 
(\ref{eq:alphaalpha'}), which states 
\beq    \sum_{i=1}^n |(\ta(i)-\ta'(i))/n| = \sum_{j=1}^{|\alpha|} \left( \sum_{i: (i, \ta(i)) \in B_j}  |(\ta(i)-\ta'(i))/n| \right) \leq \sum_{j=1}^{|\alpha|} |B_j|^2 /n. \label{eq:alphaalpha'2} \eeq

To analyze $|B_j|$, we conduct another dyadic partition similar to that in Algorithm \ref{partitionnm} with the following parameters. Let 
\beq \delta_3:=\epsilon^2\log(2/ \epsilon) / 448 \label{sec4delta3}\eeq
 be the density threshold for a square to be in $\Active$. The smallest side length of the squares is $d' = 2^{- \lceil 2 \log(1/(\epsilon))\rceil}$.  Thus $\epsilon^2/2 < d' \leq \epsilon^2$. Again, let $\Active_i$ be the set of squares of side length $2^{-i}$ and were once in $\Active$; let $\Frozen_i \subset \Frozen$ be the set of squares of side length $2^{-i}$. 
We have the following claim.
\begin{claim}\label{claim:bj}
With probability at least $1- \epsilon/2$, for any $1 \leq j \leq |\alpha| = M$, 
\[ |B_j| \leq 56 \log(2/d')\delta_3 n.\]                                                                                                                                                                        
\end{claim}
We delay the proof of this claim to later. Assuming this claim holds, we complete the proof of Lemma~\ref{tata'2}. 

Since $\sum_{j=1}^M |B_j| = n$ and by the claim, $|B_j| \leq  B^* := 56 \log(2/d')\delta_3 n$, by convexity we have
\beq
\sum_{j=1}^M |B_j|^2 / n \leq {(B^*)}^2 \cdot 1/n \cdot (n/B^*) = B^* =   56 \log(2/d')\delta_3 n.
\eeq
Therefore by (\ref{eq:alphaalpha'2}) and applying the same argument as in Lemma \ref{tata'}, we have for any $\epsilon>0$,
\begin{align*}
\pn(\ta, \ta') \leq &2 \cdot 56 \log(2/d')\delta_3   \leq 112 \log(4/\epsilon^2) \delta_3
\leq \epsilon/2.
\end{align*}  
The last inequality holds by our choice of $\delta_3$ as in (\ref{sec4delta3}). 
\end{proof}

We now finish this section by completing the proof of Claim \ref{claim:bj}
\begin{proof}[Proof of the Claim \ref{claim:bj}]
Notice that a smallest square is of side length $d'$ and can contain at most $d'n < \delta_3 n$ points of $\pi$. Hence, the smallest squares cannot be active, so  $\Active = \emptyset$. We next show that with probability at least $1-\epsilon/2$, for each $0 \leq i \leq \log(1/d')$, each square in $\Active_i$ for some $i$ contains a sample point. The reasoning is the same as in Lemma \ref{densesample}. For each square in $\Active_i$, the probability that this square does not contain a sample point is at most $(1-\delta_3)^M$. There are at most $2^{2i}$ squares in $\Active_i$; thus there are at most $\sum_{i=0}^{\log(1/d')} 2^{2i} \leq 2 \cdot d'^{-2}$ squares in $\bigcup_{i=0}^{\log(1/d')} \Active_i$. Thus, by the union bound, the probability that for every $i$, every square in $\Active_i$ contains a sample point is at least 
\[ 1- 2d'^{-2} (1-\delta_3)^M \geq 1- 2d'^{-2} e^{-\delta_3 M}  \geq 
1- \epsilon/2.\] 

We assume now that each square once in $\Active$ contains a sample point. We know $[0,1]^2$ is now partitioned into squares in $\Frozen$. Let $X = (j_t, \pi(j_t))$ be a sample point. Suppose $X$ is in $\Frozen_s$ for some $s$, which means it is in a square $S$ of side length $2^{-s}$ that is in $\Frozen$.  We show that $|\phi^{-1}(j)|$ is small by counting the number of squares in $\Frozen_{i}$ containing points of $\pi$ that are mapped to this sample point $X$ for each $i$. Note that a square in $\Frozen_i$ must have a parent square in $\Active_{i-1}$; and on the other hand, each square in $\Active_{i-1}$ contains at most four squares in $\Frozen_i$. It thus suffices to upper bound the number of squares in $\Active_{i-1}$ that contain a point of $\pi$ which maps to $X$. 

\paragraph{Case 1: $i-1 = s$.} Since each square in $\Active_{i-1}$ contains a sample point, the Euclidean distance between any point of $\pi$ in this square and the sample point in $\Active_{i-1}$ is at most $\sqrt{2} \cdot 2^{-(i-1)}$. Thus for any square $S'$ other than $S$ in $\Active_{i-1}$, if $S'$ contains a point of $\pi$ mapped to the sample point $X$ in $S$, the distance between $X$ and $S'$ is at most $\sqrt{2} \cdot 2^{-(i-1)}$. Since the disc with center $X$ and radius $\sqrt{2} \cdot 2^{-(i-1)}$ can intersect at most $14$ squares of side length $2^{-i}$ (see Figure \ref{fig:boundBtcase1}) (including $S$ itself), we know there are at most $14$ squares in $\Active_{i-1}$ containing points of $\pi$ mapped to $X$. Therefore there are at most $4 \cdot 14$ squares in $\Frozen_{i}$ containing points of $\pi$ mapped to $X$; this implies at most $4 \cdot 14 \delta_3 n$  points of $\pi$ contained in squares in $\Frozen_i$ mapped to $X$. 

\begin{figure}[h]
\centering
\caption{Case 1. \footnotesize{The yellow squares are the square $S$ containing the sample points $X$ and $X_1$ respectively. The circles with centers $X$ and $X_1$ and radius $\sqrt{2}  \cdot 2^{-(i-1)}$ respectively each intersects at most $13$ other squares in $\Active_{s}$.
}}
\includegraphics[scale=0.9]{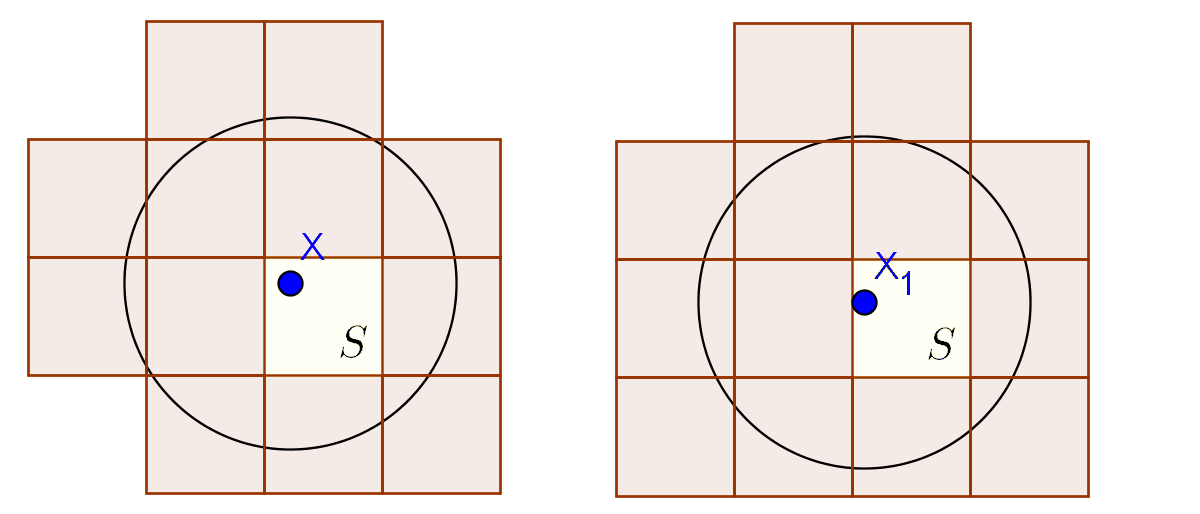}\label{fig:boundBtcase1}
\end{figure}

\paragraph{Case 2: $i-1 > s$.} Similarly, if a square $S' \in \Active_{i-1}$ contains a point mapped to the sample point $X$, then the distance from $X$ to $S'$ is at most $\sqrt{2} \cdot 2^{-(i-1)}$.  Since the disc with center $X$ and radius $\sqrt{2} \cdot 2^{-(i-1)}$ can intersect at most $12$ squares of side length $2^{-i}$ (see Figure \ref{fig:boundBtcase2}). Thus, by a similar argument, there are at most $4 \cdot 12 \delta_3 n$  points of $\pi$ contained in squares in $\Frozen_i$ mapped to $X$. 

\begin{figure}[h]
\centering
\caption{Case 2. \footnotesize{The yellow squares are the squares containing the sample points $X, X', X_1, X_1'$. The four circles with the four sample points as centers and radius $\sqrt{2}  \cdot 2^{-(i-1)}$  each intersects at most $12$ squares in $\Active_{s}$ respectively.
}}
\includegraphics[scale=0.3]{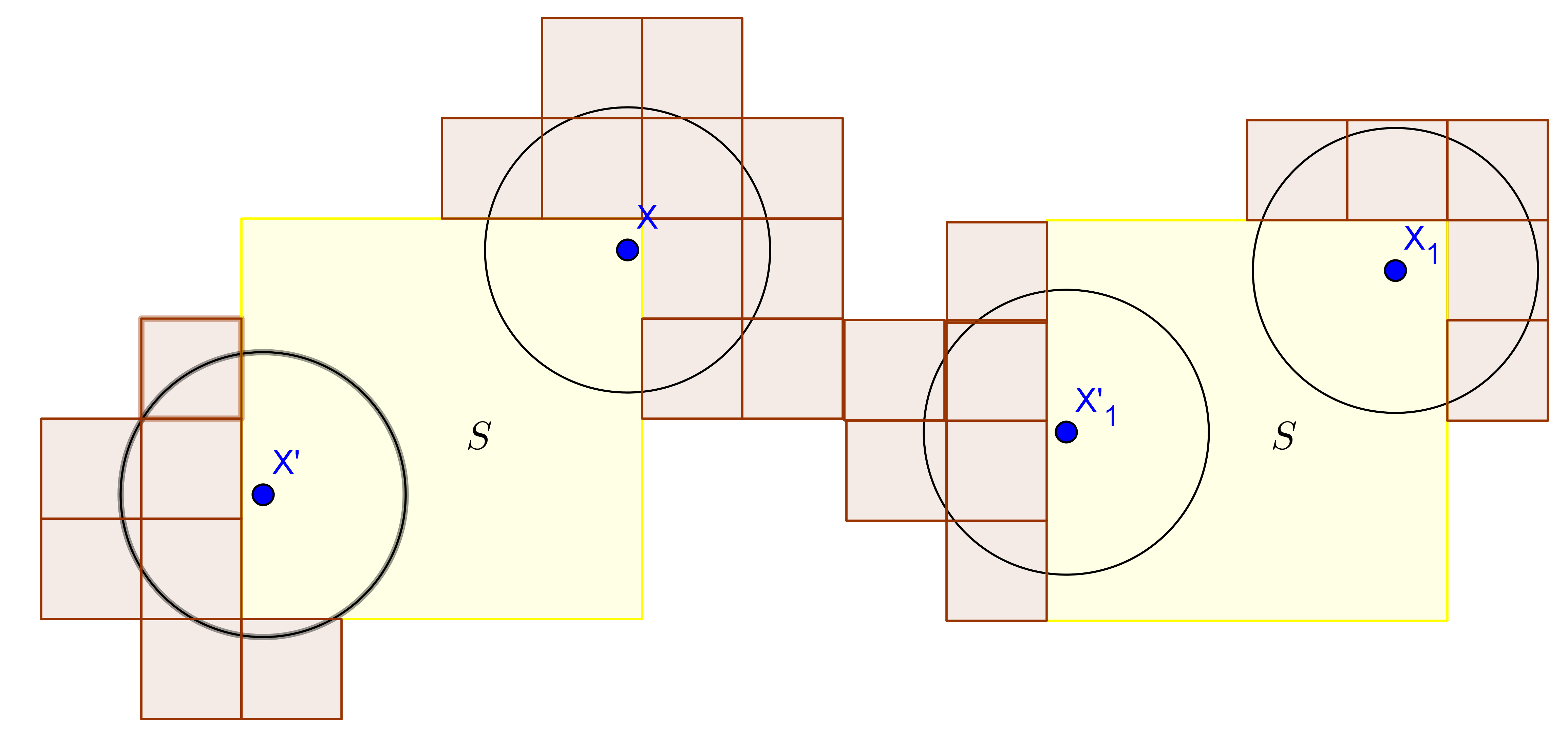}\label{fig:boundBtcase2}
\end{figure}

\paragraph{Case 3: $i-1 < s$.} By a similar argument, there are at most $4 \cdot 13 \delta_3 n$  points of $\pi$ contained in squares in $\Frozen_i$ mapped to $X$, since the disc with center $X$ and radius $\sqrt{2} \cdot 2^{-(i-1)}$ can intersect at most $13$ squares of side length $2^{-i}$ (see Figure \ref{fig:boundBtcase3}).

\begin{figure}[h]
\centering
\caption{Case 3. \footnotesize{The yellow squares are the squares $S$ containing the sample points $X, X_1, X_2, X_3$. The four circles with the four sample points as centers and radius $\sqrt{2}  \cdot 2^{-(i-1)}$ each intersects at most $13$ squares in $\Active_{s}$ respectively.
}}
\includegraphics[scale=0.7]{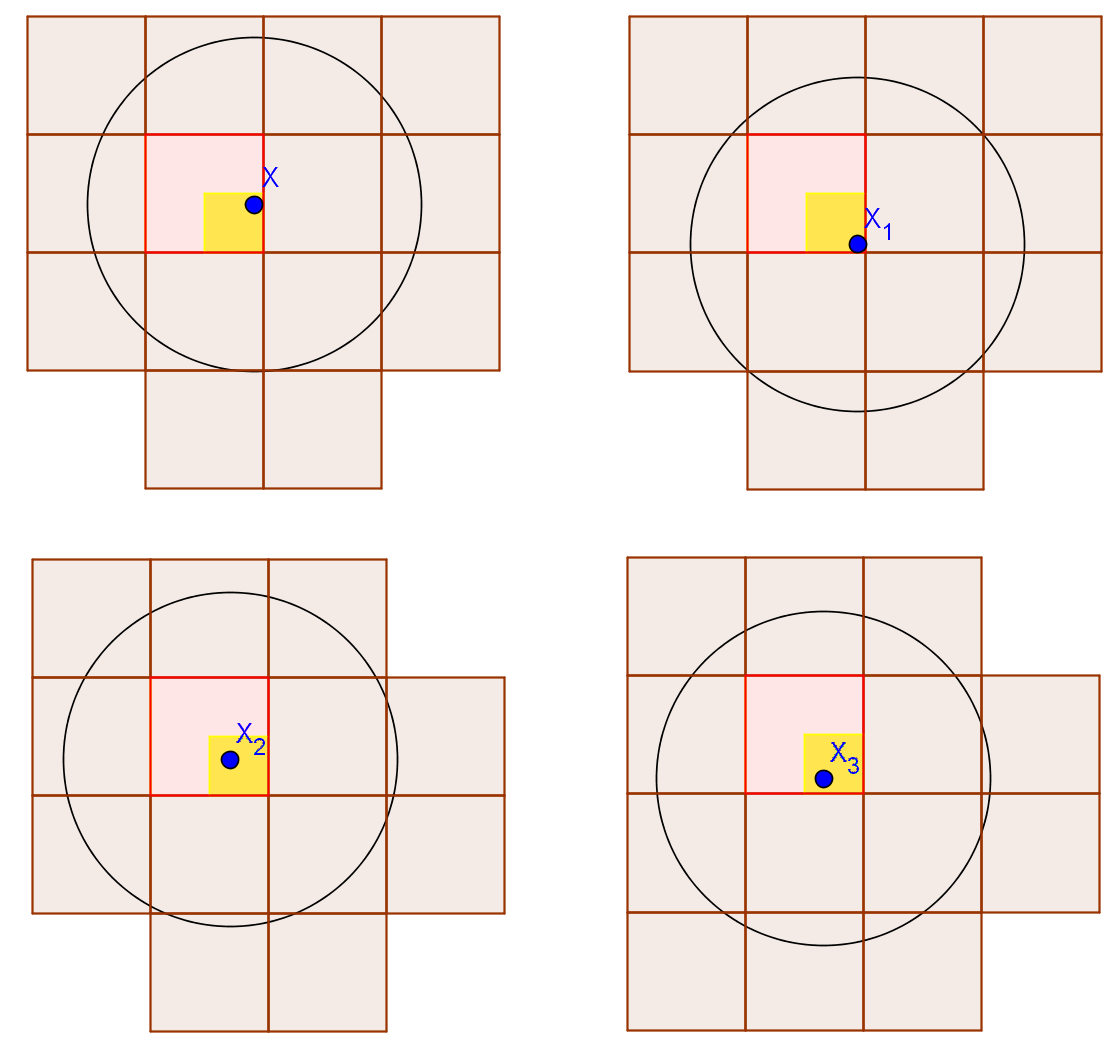}\label{fig:boundBtcase3}
\end{figure}

Combining the three cases above, we know that for each $0 \leq i \leq \log(1/d')$ and $i \neq s$, there are at most $56\delta_3 n$ points of $\pi$ contained in squares in $\Frozen_i$ mapped to the sample point $X$. And there are at most $56\delta_3 n$ points of $\pi$ in $\Frozen_s$ mapped to the sample point $X$. Therefore 
\beq
|B_j| \leq \sum_{ i = 0}^{\log(1/d')} 56\delta_3 n  = 56 \log(2/d')\delta_3 n. 
\eeq
\end{proof}

\section{One-sided property testing under the planar metrics}
We prove Theorem \ref{maincutnearlylinearoneside} in this section. 
The referee pointed out that an approach similar to that of Goldreich and Trevisan given as Proposition D.2 in \cite{GT} can be used, which shows that a (non-adaptive) two-sided tester for a hereditary property $\PP$ can be used as a black box to obtain a (canonical) one-sided tester for $\PP$. This argument simplifies our original proof and also provides a better bound. 

 We first describe the one-sided tester, which is a modification of the two-sided tester as described in Algorithm \ref{tester:1}.

With the parameters $M_1(\epsilon/2), M_2(\epsilon/2)$, and $M'$ as described in Theorem \ref{maincutnearlylinearoneside}, the tester works as follows. 
\begin{algorithm}
\caption{One-sided Tester}\label{tester:2}
Let $M'$ be as specified in Theorem \ref{newm}. Given a permutation $\pi$, we pick a subpermutation $\pi'$ of $\pi$ of size $M'$ uniformly at random. Output ``$\pi$ is in $\PP$'' if $\pi'$ is in $\PP$, and otherwise outupt ``$\pi$ is not in $\PP$''.
\end{algorithm}

\begin{proof}[Proof of Theorem \ref{maincutnearlylinearoneside}]
We first assume $32 M_1^5 k^*(M_1)/\epsilon^3 \leq 32 M_2^5 k^*(M_2)/\epsilon^3$. The other case is similar. 
If $\pi \in \PP$ then clearly the result holds since the property is hereditary. We can assume now that $\pi$ is $\epsilon$-far from $\PP$. It suffices to prove that $\text{Pr}(\pi' \in \PP) < \epsilon$.

Let $\pi''$ be a subpermutation of $\pi'$ of size $M_1(\epsilon/2)$ picked uniformly at random. Recall we have a two-sided tester in Algorithm \ref{tester:1} in Section \ref{sec:twosidedpt}; we call it Tester 1.  Given any permutation $\pi'$, applying Tester \ref{tester:1} with parameter $\epsilon/2$, we pick a subpermutation $\pi''$ of size $M_1(\epsilon/2)$ of $\pi'$ uniformly at random. For simplicity, we say ``$\pi'$ is accepted by Tester 1" if Tester 1 outputs ``$\pi' \in \mathcal{P}$.
Therefore, for $\pi'$ a subpermutation of $\pi$ of size $M'$ picked uniform at random from $\pi$, we have 
\[
\text{Pr}(\pi' \text{ is accepted by Tester \ref{tester:1}}) \geq \text{Pr}(\pi'  \text{ is accepted by Tester \ref{tester:1}} | \pi' \in \PP) \text{Pr}(\pi' \in \PP) \geq (1-\epsilon/2)  \text{Pr}(\pi' \in \PP).
\]
The second inequality holds by a direct consequence of Theorem \ref{maincutnearlylinear}. Since the size of $\pi'$ is at least the value of $n_0(M_1(\epsilon/2), \epsilon/2)$ in Theorem \ref{maincutnearlylinear} and $\pi' \in \PP$, a subpermutation $\pi''$ of size $M_1(\epsilon/2)$ of $\pi'$ will be accepted with probability at least $1-\epsilon/2$. 

Furthermore, since $\pi$ is $\epsilon$-far from $\PP$ and is of size at least  the value $n_0(M_1(\epsilon/2), \epsilon/2)$ in Theorem \ref{maincutnearlylinear}, and $\pi''$ is a subpermutation of size $M_1(\epsilon/2)$ of $\pi$ picked uniformly at random, Theorem \ref{maincutnearlylinear} tells us 
\[\text{Pr}(\pi' \text{ is accepted by Tester \ref{tester:1}}) \leq \epsilon/2.\]
Therefore we have 
\[  \text{Pr}(\pi' \in \PP) \leq (\epsilon/2) / (1-\epsilon/2) < \epsilon.\]

\end{proof}

\noindent {\bf Acknowledgement.} The authors would like to thank the referees for many valuable comments. In particular, one of the referees pointed out the  simplified proof for Theorem \ref{maincutnearlylinearoneside}.

\end{document}